\documentclass[11pt]{amsart}

\usepackage[utf8]{inputenc}
\usepackage{amstext,latexsym,amsbsy,amsmath,amssymb,amsthm,mathtools,relsize,geometry,enumerate}
\usepackage{hyperref}
\hypersetup{
  colorlinks   = true, 
  urlcolor     = blue, 
  linkcolor    = red, 
  citecolor   = red 
}
\theoremstyle{definition}
\usepackage{graphicx}
\setkeys{Gin}{width=\linewidth,totalheight=\textheight,keepaspectratio}
\graphicspath{{./images/}}

\usepackage{multicol}
\usepackage{booktabs}
\usepackage{mathrsfs}\usepackage{color}

\newcommand{\nbb}{\mathbb{N}}
\newcommand{\rbb}{\mathbb{R}}

\renewcommand{\L}{\mathcal{L}}

\newcommand{\W}{\mathcal{W}}
\renewcommand{\H}{\mathcal{H}}
\newcommand{\PMarkov}{\mathcal{P}}
\newcommand{\B}{\mathcal{B}}
\newcommand{\la}{\langle}
\newcommand{\ra}{\rangle}

\newcommand{\Xtil}{\widetilde{X}}
\newcommand{\Xbar}{\overline{X}}

\newcommand{\xbar}{\overline{x}}
\newcommand{\vbar}{\overline{v}}
\newcommand{\zbar}{\overline{z}}

\newcommand{\Xbf}{\mathbf{X}}
\newcommand{\Ubf}{\mathbf{U}}
\newcommand{\Vbf}{\mathbf{V}}

\newcommand{\mi}{\wedge}
\newcommand{\D}{\mathcal{D}}
\newcommand{\Tr}{\text{Tr}}

\newcommand{\dW}{\text{d}W}

\newcommand{\f}{\varphi}

\newcommand{\E}[1]{\mathbb{E}\left[#1\right]}

\newcommand{\Enone}[1]{\mathbb{E}#1}

\renewcommand{\P}[1]{\mathbb{P}\left\{#1\right\}}
\newcommand{\Pnone}{\mathbb{P}}

\newcommand{\wt}[1]{ \widetilde{#1} }

\newcommand{\Hs}{\mathcal{H}_{-s}}

\theoremstyle{definition}
\theoremstyle{plain}
\newtheorem{theorem}{Theorem}
\newtheorem{corollary}[theorem]{Corollary}

\newtheorem{lemma}[theorem]{Lemma}
\newtheorem{assumption}[theorem]{Assumption}
\newtheorem{proposition}[theorem]{Proposition}

\newtheorem{remark}[theorem]{Remark}

\numberwithin{equation}{section}

\title[Generalized Langevin Equation with Power-Law Memory]{The Generalized Langevin Equation with power-law memory in a nonlinear potential well}

\author[N.E.~Glatt-Holtz]{Nathan E.~Glatt-Holtz$^1$}
\author[D.P.~Herzog]{David P.~Herzog$^2$}
\author[S.A.~McKinley]{Scott A.~McKinley$^1$}
\author[H.D.~Nguyen]{Hung D.~Nguyen$^3$}

\thanks{\noindent \hspace{-0.52cm} $^1$ Department of Mathematics, Tulane University. 6823 St Charles Ave, New Orleans, LA, 70118.\\
$^2$ Department of Mathematics, Iowa State University. 411 Morrill Rd, Ames, Iowa 50011.\\
$^3$ Department of Mathematics, University of California, Los Angeles. 520 Portola Plaza, Los Angeles, CA 90095}

\calclayout

\begin{document}

\maketitle


\begin{abstract}
The generalized Langevin equation (GLE) is a stochastic integro-differential equation that has been used to describe the velocity of microparticles in viscoelastic fluids. In this work, we consider the large-time asymptotic properties of a Markovian approximation to the GLE in the presence of a wide class of external potential wells. The qualitative behavior of the GLE is largely determined by its memory kernel $K$, which summarizes the delayed response of the fluid medium on the particles past movement. When $K$ can be expressed as a finite sum of exponentials, it has been shown that long-term time-averaged properties of the position and velocity do not depend on $K$ at all. In certain applications, however, it is important to consider the GLE with a power law memory kernel. Using the fact that infinite sums of exponentials can have power law tails, we study the infinite-dimensional version of the Markovian GLE in a potential well. In the case where the memory kernel $K$ is integrable (i.e. in the asymptotically diffusive regime), we are able to extend previous results and show that there is a unique stationary distribution for the GLE system and that the long-term statistics of the position and velocity do not depend on $K$.  However, when $K$ is not integrable (i.e. in the asymptotically subdiffusive regime), we are able to show the existence of an invariant probability measure but uniqueness remains an open question. In particular, the method of asymptotic coupling used in the integrable case to show uniqueness does not apply when $K$ fails to be integrable. 
\end{abstract}

\vspace{0.5cm}


\section{Introduction}

The movement of microparticles in biological fluids is often distinct from classical Brownian motion \cite{hofling2013anomalous}. While some particles exhibit non-Gaussian~\cite{wang2012brownian, lampo2017cytoplasmic} and/or switching behavior \cite{metzler2014anomalous, newby2017blueprint}, an important category of \emph{anomalous diffusion} includes paths with stationary Gaussian increments that are negatively correlated with each other \cite{kou2008stochastic, ernst2012fractional, weber2012analytical, hill2014biophysical}.  In particular, anticorrelation can accumulate in such a way that the particle process $\{x(t)\}_{t \geq 0}$ has a mean squared displacement (MSD, $\E{|x|^2(t)}$) that is sublinear over a significant period of time \cite{morgado2002relation, kupferman2004fractional,mckinley2017anomalous}. Such a process is commonly called \emph{subdiffusive}, and there have been numerous perspectives on how to model \cite{mason1995optical, kou2004general,kupferman2004fractional, mckinley2009transient, indei2012treating, goychuk2012viscoelastic,hohenegger2017fluid} and statistically analyze \cite{kou2008stochastic,fricks2009time, cordoba2012elimination, meroz2013test, lysy2016model,hohenegger2018reconstructing} individual, unconstrained, subdiffusive microparticle paths. 

It is also natural to investigate the behavior of subdiffusive particles when they are subjected to external forces. Commonly the magnitude and direction of these forces are expressed as the negative gradient of a space-dependent potential energy function $\Phi : \mathbb{R}^d \to \mathbb{R}$. In early works, it was common to study quadratic potentials ($\Phi(x) = \kappa |x|^2$)  because the equations that govern such particle motion are linear and admit exact solutions \cite{adelman1976fokker,kupferman2004fractional,kou2004general}. More recently there has been success in modeling and simulating the behavior of subdiffusive particles in nonlinear potentials as well. As we discuss below, it is possible to develop a model that exhibits ``transient subdiffusive behavior'' (sublinear MSD over several log-decades of time \cite{mckinley2009transient,mckinley2017anomalous}) in the absence of a confining potential, and permits an exact expression for the stationary joint distribution of position and velocity when subjected to one of an appropriate class of nonlinear potentials \cite{ottobre2011asymptotic,goychuk2012viscoelastic,pavliotis2014stochastic}. There are a host of classical questions that can be rigorously studied for Brownian motion in such potentials, which provides a foundation to ask similar questions for subdiffusive particles. Several such questions were introduced and surveyed by I.~Goychuk in 2009 \cite{goychuk2009viscoelastic} and 2012 \cite{goychuk2012viscoelastic}: for example, addressing escape times between minima in double-welled potentials (or more recently, escape times in stationary Gaussian potentials \cite{goychuk2018viscoelastic}), effective diffusivity in static periodic potentials, effective velocity induced by time-dependent potentials known as flashing or rocking rachets (see also \cite{goychuk2010subdiffusive,goychuk2012fractional,kharchenko2012flashing,goychuk2013rocking,kharchenko2013subdiffusive}). Since organelles and other microparticles (or macromolecules) have been observed to exhibit significant subdiffusive behavior in cytoplasm \cite{tolic2004anomalous, weiss2004anomalous,weber2010bacterial} and other biological fluids like mucus \cite{hill2014biophysical}, applications of these theoretical models have emerged in recent years: for example in the study of magnetic nanoparticles \cite{goychuk2015modeling,goychuk2018sensing} and the modeling of intracellular transport by microtubule-associated molecular motors \cite{goychuk2014molecular,goychuk2014moleculartoo,goychuk2015anomalous}. Moreover, it is commonplace for experimentalists to probe fluid-mechanical properties of live cells by manipulation through ``optical tweezers'' \cite{neuman2008single,wei2008comparative,zhang2008optical}. However, such studies rarely take into account the subdiffusive character of the microparticle probes, or the possibly nonlinear forces exerted by the trap.

In this work, we take a step advancing the theory for a set of models that are ``fully subdiffusive'' in a sense described below.  Such models are necessarily infinite-dimensional, which makes it non-trivial to establish existence and uniqueness of stationary measures, and points towards fundamental questions about if and on what time scale a system might ``forget'' its initial conditions.


\subsection{Model development and the main question}

To describe microparticle motion in \emph{viscous} fluids, it is common to use a Langevin framework. Let $\{(x(t),v(t))\}_{t \geq 0}$ denote the position and velocity of a particle, and let $\Phi(x)$ denote the particle's potential energy at the position $x$. (Because it does not have a substantive impact on our results, we will study the dynamics in one dimension.) Newton's Second Law yields \cite{pavliotis2014stochastic}
\begin{equation} \label{eq:langevin}
	m \, d v(t) = -\gamma v(t) - \Phi'(x(t)) dt + \sqrt{2\gamma} d W(t),
\end{equation}
where $x'(t) = v(t)$. Here $m$ is the particle's mass, $\gamma$ is the viscous drag coefficient, and $W(t)$ is a standard Brownian motion. To be physically correct, the coefficient of the noise should be $\sqrt{2 k_B T \gamma}$, where $k_B$ is Boltzman's constant and $T$ is the temperature, but for the sake of notational simplicity we will set $k_B T = 1$ throughout.  Under appropriate conditions on the potential well, this system has a unique stationary distribution with density 
\begin{equation} \label{eqn:marginal-x-v}
	\pi(x,v) \propto \exp\Big(\!\!- \big(\Phi(x) + \frac{m}{2} v^2\big)\Big)
\end{equation}
and is geometrically ergodic, in the sense that the law of the process converges to the stationary distribution exponentially quickly (see, for example, \cite{mattingly2002ergodicity, cooke2017geometric, herzog2017ergodicity, pavliotis2014stochastic} as well as \cite{conrad2010construction, grothaus2015hypocoercivity} for related results). Birkoff's ergodic theorem in turn implies
\begin{equation} \label{eqn:ergodic-theorem}
	\lim_{t \to \infty}\frac{1}{t} \int_0^t f(x(s), v(s)) ds = \int_{\rbb^2} f(x,v) \pi(x,v) dx dv, \quad \pi\text{-a.s. and in }L^1(\pi)
\end{equation}
for any $f \in L^1(\pi)$.  Ultimately, we are interested in whether or not such a property holds for viscoelastic diffusion in a nonlinear potential well. 

In order to model the memory effects arising in viscoelastic diffusion, physicists have long employed the generalized Langevin equation (GLE), which adds a memory kernel $K: \rbb \to \rbb_+$ and a stationary Gaussian process $\{F(t)\}_{t \in \rbb}$ to the Langevin dynamics \eqref{eq:langevin}. In particular, the GLE in a potential well $\Phi$ can be written formally as
\begin{equation}\label{eqn:GLE}
m\, \dot{v}(t)=- \Phi'(x(t)) -\gamma v(t) - \int_{-\infty}^t K(t-s)v(s)\, ds + \sqrt{2\gamma} \, \dot{W}(t) + F(t).
\end{equation}
for $t \geq 0$, where we assume that $\E{F(t) F(s)} = K(t-s)$ in order to satisfy the fluctuation-dissipation relation~\cite{kubo1966fluctuation,hohenegger2017equipartition,hohenegger2018reconstructing}. In general, there can be a pair of coefficients in front of the memory terms, but they do not affect our analysis. We refer the reader to \cite{hohenegger2017fluid} for a physical interpretation of those parameters. In comparison with the classical GLE version in \cite{kubo1966fluctuation}, the noise term in~\eqref{eqn:GLE} can be regarded as a sum of two independent noises $\xi(t)=F(t)+\sqrt{2\gamma}\dot{W}(t),$ such that $\mathbb{E}[F(t)\dot{W}(t)]=0$, and $\mathbb{E}[\sqrt{2\gamma}\dot{W}(t) \sqrt{2\gamma}\dot{W}(t')]=2\gamma\delta_0(|t-t'|)$. In this case, the memory term $\gamma(t)$ is a sum of the Dirac delta function and $K(t)$, i.e. $\gamma(t)=\gamma\delta(t)+K(t)$ and thus
$$\int_{-\infty}^t \gamma(t-s)v(s)\, ds=\gamma v(t)+\int_{-\infty}^t K(t-s)v(s)\, ds.$$
Separation of $F(t)$ and $\dot{W}(t)$ will help with the analysis later. We also note a distinction in the support of the memory kernel. Namely, in the physics literature, the memory kernel in the GLE is often defined on the interval $[0,t]$ rather than $(-\infty,t]$ as we have here. In order for a stationary integro-differential equation to be properly defined we must specify initial data on an infinite horizon. The classical definition effectively defines the velocity to be zero on the interval $(-\infty, 0]$. This assumption (along with a quadratic potential well), allow for use of the Laplace transform, which can yield a host of statistical properties including an exact solution \cite{vinales2006anomalous,desposito2008memory}. Unfortunately, such an approach is not available here.

Due to the memory kernel $K$, in this formulation of GLE, the joint process $(x(t),v(t))$ is non-Markovian.  Therefore, it is not immediately clear what we mean by a ``stationary distribution.'' However, if the memory kernel $K(t)$ is a sum of $N$ exponential functions, we can use the so-called Mori-Zwanzig formalism \cite{zwanzig2001nonequilibrium, goychuk2009viscoelastic, ottobre2011asymptotic} to rewrite the GLE as a $2+N$ (or $2 + 2N$ \cite{fricks2009time}) dimensional system of SDEs. (For a contrast on the two representations, see \cite{hohenegger2018reconstructing}.) When $\Phi$ is quadratic, such a representation is statistically equivalent to \eqref{eqn:GLE}; otherwise, we simply call this the Markovian approximation of the GLE. This  finite-dimensional Markovian version of the GLE does admit a stationary distribution and one can show that the marginal distribution of the pair $(x,v)$ in stationarity is exactly \eqref{eqn:marginal-x-v} \cite{ottobre2011asymptotic,goychuk2012viscoelastic,pavliotis2014stochastic}.  The fact that the memory kernel does not affect the stationary statistics of $x$ and $v$ is, in some sense, a natural generalization of the observation that the drag coefficient $\gamma$ does not appear in $\pi(x,v)$ for viscous diffusion. It is then reasonable to ask whether this property holds for more general forms of $K$.

The sum-of-exponential form for $K$ is a very useful construct, but it turns out that restricting ourselves to \emph{finitely} many terms neglects an important qualitative regime. Indeed, if $K \in L^1(\rbb)$, then the associated unconstrained GLE ($\Phi \equiv 0$) is always \emph{asymptotically diffusive} in the sense that $\E{x^2(t)} \sim t$ as $t \to \infty$ \cite{ottobre2011asymptotic, mckinley2017anomalous}.  Here we write $f(t) \sim g(t)$ as $t \to \infty$ if $\lim_{t \to \infty} f(t)/g(t) = C \in (0,\infty)$.  However, if $K \notin L^1(\rbb)$ but $K(t) \sim t^{-\alpha}$ for some $\alpha \in (0,1)$ as $t \to \infty$, then under mild restrictions, the unconstrained GLE is \emph{asymptotically subdiffusive}, i.e.~$\E{x^2(t)} \sim t^\alpha$ \cite{mckinley2017anomalous}. Moreover, a new \emph{critical} regime is recently found when $K(t)\sim t^{-1}$, for which case, $\mathbb{E}[x^2(t)]\sim t/\log t$ \cite{didier2019asymptotic}. See also \cite{kou2008stochastic,pottier2003aging,sandev2019fractional} for related results. We are primarily interested in memory kernels with power law tails, which can fall in either qualitative regime. As has been observed elsewhere~\cite{abate1999infinite,goychuk2009viscoelastic}, it is possible for an \emph{infinite} sum of exponentials to have a power law tail. Therefore, an infinite-dimensional version of the Mori-Zwanzig formalism is an appropriate way to study the GLE with power law memory.

In this work, we explore the infinite-dimensional Markovian version of the GLE with an eye toward addressing the fundamental question of whether $(x(t),v(t))$ is ergodic in the sense of \eqref{eqn:ergodic-theorem}. In Section \ref{sec:results} we introduce notation, explicitly define our model, and summarize the main results. In Section \ref{sec:well-posed} we establish well-posedness for all values of $\alpha > 0$ (including both the asymptotically diffusive and subdiffusive cases). Using an extension of the invariant measure previously established for the finite-dimensional GLE, in Section \ref{sec:invariant-exist}, we demonstrate the existence of an explicitly defined invariant probability measure in the infinite-dimensional case, cf. Theorem~\ref{thm:density}. In Section \ref{sec:invariant-unique}, we use \emph{asymptotic coupling} \cite{weinan2001gibbsian,mattingly2002exponential, hairer2011asymptotic} to establish uniqueness of this measure in the asymptotically diffusive case ($\alpha>1$) under a wide class of nonlinear potentials including those polynomials of even degree, see Assumption \ref{cond:Phi-1} and Theorem~\ref{thm:inv-measure:unique}. We finish with Section \ref{sec:discussion} which contains conclusions and discussions of open and related problems. In particular, we cannot extend our proof of unique ergodicity to the asymptotically nondiffusive case ($\alpha\in(0,1]$). We have yet to determine whether this is a shortcoming of current methods, or if the claim is simply not true. We hope that this work will shine some light on this interesting question.

\section{Notation and Rigorous Summary of Results}

\label{sec:results}

Suppose that a memory kernel $K(t)$ has the form
\begin{equation}\label{eqn:K-finite-mode}
K(t)=\sum_{k=1}^N c_k e^{-\lambda_k t},
\end{equation}
where $c_k,\ \lambda_k>0$, $k=1,\dots,N$. Then, following Chapter 8 of~\cite{pavliotis2014stochastic}, which summarizes pre-existing work on the same topic \cite{ottobre2011asymptotic,goychuk2012viscoelastic}, we can use Duhamel's formula and set
\begin{align*}
z_k(t)= e^{-\lambda_k t} z_k(0) + \sqrt{c_k} \int_0^t e^{-\lambda_k(t-s)} v(s) \, ds + \sqrt{2\lambda_k} \int_0^t e^{-\lambda_k (t-s) } \, dW_k(s).
\end{align*}
In order to approximate equation~\eqref{eqn:GLE} as an $(N+2)$-dimensional Markov system, we write
\begin{equation} \label{eqn:finite-mode}
\begin{aligned}
d x(t) &= v(t) \, d t,\\
m\, d v(t)&=\Big(-\gamma v(t)-\Phi'(x(t))-\sum_{k=1}^N \sqrt{c_i} z_i(t)\Big)dt+\sqrt{2\gamma}\, dW_0(t),\\
d z_k(t)&=(-\lambda_k z_k(t)+ \sqrt{c_k}v(t))\, dt+\sqrt{2\lambda_k}\, dW_k(t),\qquad 1\leq k\leq N,\\
\end{aligned}
\end{equation}
where $z_k(0)$ are i.i.d. $\mathcal{N}(0,1)$ random variables and $(W_0,W_1,\dots,W_N)$ is a standard $(N+1)$-dimensional Brownian motion. This is the content of Proposition 8.1 of~\cite{pavliotis2014stochastic}.  It is also known (see~Proposition 8.2, \cite{pavliotis2014stochastic}) that the system \eqref{eqn:finite-mode} is uniquely ergodic with an invariant probability density function $\varrho(x,v,z_1,\dots,z_N)$ given by
\begin{equation} \label{eqn:density-finite-mode}
\varrho(x,v,z_1,\dots,z_N) \propto\exp\Big\{-\Phi(x)-\frac{m}{2}v^2-\frac{1}{2}\sum_{k=1}^N  z_k^2\Big\}.  
\end{equation}

As discussed above, in order to study power law memory kernels we consider infinite sums of exponential functions. To this end, let $\alpha, \beta >0$ be given, and define constants $c_k, \lambda_k$, $k=1,2, \ldots$, by 
\begin{align} \label{c-k}
c_k=\frac{1}{k^{1+\alpha\beta}},\ \lambda_k=\frac{1}{k^\beta}.
\end{align}   
Define the kernel $K$ by
\begin{equation}
\label{eqn:K}
K(t)=\sum_{k\geq 1} c_k e^{-\lambda_k t}.
\end{equation}
It follows (see Example 3.3 of~\cite{abate1999infinite}) with this choice of constants $c_k$ and $\lambda_k$ that
\begin{equation} \label{lim:K}
K(t) \sim t^{-\alpha} \,\,\, \text{ as } \,\,\, t\rightarrow \infty, 
\end{equation}
thus giving the desired power law tail for the memory kernel $K$. This definition of the constants and exponents is not unique in yielding the power-law memory and subdiffusive behavior. For an alternate (and more efficient in the sense that few terms can yield a power law over similar time scales), see \cite{goychuk2009viscoelastic} and \cite{goychuk2012viscoelastic} for some discussion.

With this definition for $K$, we consider the following infinite-dimensional system of stochastic differential equations 
\begin{equation}\label{eqn:GLE-Markov}
\begin{aligned}
d x(t) &= v(t)\, d t, \\
m\, d v(t)&=\big(-\gamma v(t)-\Phi'(x(t))-\sum_{k\geq 1} \sqrt{c_k} z_k(t)\big)\,dt+\sqrt{2\gamma}\, dW_0(t) ,\\
d z_k(t)&=\left(-\lambda_k z_k(t)+ \sqrt{c_k}v(t)\right) \, dt+\sqrt{2\lambda_k}\, dW_k(t),\qquad k\geq 1, 
\end{aligned}	
\end{equation}
where the $W_k$ are independent, standard Brownian motions. Throughout this work, we will assume that the potential $\Phi$ satisfies the following growth and regularity conditions:
\begin{assumption}\label{cond:Phi} $\Phi\in C^\infty(\rbb)$, $\int_\rbb|\Phi'| e^{-\Phi}dx<\infty$ and there exists a constant $b>0$ such that for all $x\in\rbb$
\begin{align*}
\quad b(\Phi(x)+1)\geq x^2.     
\end{align*}
\end{assumption}

A typical class of potentials $\Phi$ that satisfies Assumption~\ref{cond:Phi} is the class of polynomials of even degree whose leading coefficient is positive.
\begin{remark}
The first and third parts of Assumption~\ref{cond:Phi} are quite standard, giving nominal regularity as well as assuring the potential grows at least as fast as a quadratic at infinity.  The second condition is also a nominal growth condition on the derivative of $\Phi$ and will be used in Section~\ref{sec:invariant-exist} to check that our candidate invariant measure is indeed invariant.    
\end{remark}

In order to define a phase space for the infinite-dimensional process 
\begin{align*}
X(t)=(x(t), v(t), z_1(t), z_2(t), \ldots), 
\end{align*}
we will make use of the Hilbert space $\H_{-s}$, $s\in \rbb$, equipped with the inner product $\la\cdot,\cdot\ra_{-s}$, 
 \begin{equation}\label{eqn:H_p}
\H_{-s}=\Big\{X=(x,v,z_1, z_2, \ldots):x^2+v^2+\sum_{k\geq 1}k^{-2s}z_k^2<\infty\Big\},
\end{equation}
and 
\begin{equation} \label{eqn:H-inner-prod}
\la X,\widetilde{X}\ra_{-s} = x\wt{x}+v\wt{v}+\sum_{k\geq 1}k^{-2s}z_k\wt{z}_k.
\end{equation}
We denote by $\| \cdot \|_{-s}$ the norm in $\H_{-s}$ given by 
\begin{align}\label{eqn:H-norm}
\|X\|_{-s}^2=x^2+v^2+\sum_{k\geq 1}k^{-2s}z_k^2. 
\end{align}
The canonical basis $\mathcal{D}=\{e_x,e_v,e_1,e_2,\dots\}$ in $\H_{-s}$ is then given by
\begin{equation} \label{form:H-basis}
\begin{aligned}
e_x&=(1,0,0,0,\dots),\\
e_v &=(0,1,0,0,\dots),\\
e_k & = (0,0,0,\dots,k^s,0,\dots),\, k\geq 1.
\end{aligned}
\end{equation}
From now on, for simplicity, we omit the subscript $-s$ in the inner product $\la\cdot,\cdot\ra$. In view of~\eqref{form:H-basis}, for $X=(x,v,z_1,z_2,\dots)$, $X$ we may write 
\begin{equation}\label{eqn:density-5}
X=\la X,e_x\ra e_x+\la X,e_v\ra e_v+\sum_{k\geq 1}\la X,e_k\ra e_k=  xe_x+ve_v+\sum_{k\geq 1}k^{-s}z_ke_k.
\end{equation}

Next, we collect several formulas involving Fr\'echet derivatives that will be useful later, especially in Section~\ref{sec:invariant-exist}. For $\psi:\Hs\to\rbb$, let $D\psi:\Hs\to\Hs$ be the first Fr\'echet derivative, if it exists. Then, the derivative of $\psi$ in the direction of $e\in\D$ is given by
\begin{displaymath}
\la D\varphi(X),e\ra=\lim_{\varepsilon\to 0}\frac{\varphi(X+\varepsilon e)-\varphi(X)}{\varepsilon}=\frac{\partial\psi}{\partial\la X,e\ra}(X).
\end{displaymath}
In view of representation~\eqref{eqn:density-5}, substituting $e$ with $e_x$, $e_v$, $e_k$, $k\geq 1$ in the above formula gives
\begin{equation} \label{eqn:Frechet:first}
\la D\psi(X),e_x\ra =\frac{\partial\psi}{\partial x}(X ),\qquad \la D\psi(X),e_v\ra =\frac{\partial\psi}{\partial v}(X ),
\end{equation}
\begin{equation*}
\text{and    }
\la D\psi(X),e_k\ra = k^s\frac{\partial\psi\left(X \right)}{\partial z_k}
\end{equation*}
Similarly, if $\psi$ is twice Fr\'echet differentiable, let $ D^2\psi:\Hs\to L(\Hs,\Hs)$ be the second Fr\'echet derivative, where $L(\Hs,\Hs)$ denotes the space of linear bounded maps from $\Hs$ to itself. Then, for $e\in\D$, we have
\begin{displaymath}
\la  D^2\psi(X)(e),e\ra=\lim_{\varepsilon\to 0}\frac{1}{\varepsilon}\la  D\psi(X+\varepsilon e)- D\psi(X),e\ra=\frac{\partial^2\psi}{\partial\la X,e\ra^2}(X).
\end{displaymath}
Thus
\begin{equation}\label{eqn:Frechet:second}
\la D^2\psi(X)(e_x),e_x\ra=\frac{\partial^2\psi}{\partial x^2}\left(X \right), \qquad \la D^2\psi(X)(e_v),e_v\ra=\frac{\partial^2\psi}{\partial v^2}\left(X \right),
\end{equation}
\begin{equation*}
\text{and  } \la D^2\psi(X)(e_k),e_k\ra=k^{2s}\frac{\partial^2\psi}{\partial z_k^{2}}\left(X \right).
\end{equation*}

Throughout, unless otherwise stated, we make the following assumptions about kernel parameters $\alpha, \beta$ as in~\eqref{c-k} and \eqref{eqn:K} and the phase space regularity parameter $s$.  
\begin{assumption}\label{cond:sol'n} Let $\alpha,\beta>0$ be as in~\eqref{c-k} and $s$ as in~\eqref{eqn:H_p}. We assume that they satisfy either the \emph{asymptotically diffusive} condition  
\begin{enumerate}
\item[\emph{(D)}]\label{cond:diffusion} $\displaystyle \alpha>1,\, \beta>\frac{1}{\alpha-1}$ and $\displaystyle \frac{1}{2}< s < \frac{(\alpha-1)\beta}{2}$;
\end{enumerate}  or the \emph{critical} condition
\begin{enumerate}
\item[\emph{(C)}] $\displaystyle \alpha =1,\, \beta>1$ and $\displaystyle \frac{1}{2}< s< \frac{ \beta}{2};$
\end{enumerate}
or the \emph{asymptotically subdiffusive} condition 
\begin{enumerate}
\item[\emph{(SD)}]\label{cond:subdiffusion} $\displaystyle 0<\alpha <1,\, \beta>\frac{
1}{\alpha}$ and $\displaystyle \frac{1}{2}< s< \frac{\alpha \beta}{2}.$
\end{enumerate}
\end{assumption} 
\begin{remark} 
The assumption above really concerns the parameters $\alpha, \beta$ only.  Indeed, the particular choice of $s$ in either part is the natural phase space range for the process for those particular choices of $\alpha, \beta$.  It is also worth remarking that, so long as $\beta>0$ is large enough, the above simply splits the dynamics in the diffusive $(\alpha >1)$ and the other two $(0<\alpha<1,\,\text{and}\,\, \alpha=1)$ regimes.  
\end{remark}

\begin{remark} In our context, to relate the infinite-dimensional system~\eqref{eqn:GLE-Markov} to the original equation~\eqref{eqn:GLE}, the initial data $z_k(0)$ for~\eqref{eqn:GLE-Markov} is necessarily i.i.d. with $\mathcal{N}(0,1)$ distribution. To ensure consistency, we want $(z_1(0),z_2(0),\dots)$ to live in $\H_{-s}$ almost surely. This means that we must restrict to the phase space $\H_{-s}$ for $s>1/2$. It turns out that this constraint on $s$ is also required to compute the density of the invariant measure, see Section~\ref{sec:invariant-exist}.
\end{remark} 

Fixing a stochastic basis $\mathcal{S}=\left(\Omega,\mathcal{F},\Pnone,\{\mathcal{F}_t\}_{t\geq 0},W\right)$ satisfying the usual conditions, cf. \cite{karatzas2012brownian}, where $W$ is the cylidrical Wiener process defined later in~\eqref{eqn:Wiener}, we state the following result giving existence and uniqueness of solutions of~\eqref{eqn:GLE-Markov}.
\begin{proposition} \label{prop:sol'n} Suppose that $\Phi$ satisfies Assumption \ref{cond:Phi} and the constants $\alpha, \beta, s$ satisfy Assumption \ref{cond:sol'n}. Then for all initial conditions $X_0\in \H_{-s}$, there exists a unique pathwise solution $X(\cdot,X_0):\Omega\times[0,\infty)\to\H_{-s}$ of~\eqref{eqn:GLE-Markov} in the following sense: $X(\cdot,X_0)$ is $\mathcal{F}_t$-adapted, $X(\cdot,X_0)\in C([0,\infty);\H_{-s})$ almost surely and that if $\wt{X}(\cdot,X_0)$ is another solution then for every $T\geq 0$,
\begin{displaymath}
\P{\forall t\in[0,T], X(t,X_0)=\wt{X}(t,X_0)}=1.
\end{displaymath}
Moreover, for every $X_0 \in \H_{-s}$ and $T\geq 0$, there exists $C(T,X_0)>0$ such that
 \begin{equation} \label{ineq:strong-sol'n}
 \E{\sup_{0\leq t\leq T }\|X(t)\|_{-s}^2}\leq C(T,X_0),
 \end{equation}
and that for any bounded set $B\subset\H_{-s}$, we have 
\begin{equation}\label{ineq:strong-sol'n:sup}
\sup_{X_0\in B}C(T,X_0)<\infty.
\end{equation}
\end{proposition}
The proof of Proposition~\ref{prop:sol'n} will be carried out in Section~\ref{sec:sol'n}.

Our candidate stationary measure for the system~\eqref{eqn:GLE-Markov} is an infinite-dimensional analogue of the one defined in~\eqref{eqn:density-finite-mode}.  To write it down, let $\mu_x, \mu_v$ and $\nu$ denote the probability measures on $\rbb$ defined by 
\begin{equation} \label{defn:mu-x}
\mu_x(dy)=\frac{1}{\int_\rbb e^{-\Phi(z)} \, dz} e^{-\Phi(y)} \, dy, \,\,\, \mu_v(dy) = \frac{\sqrt{m}}{\sqrt{2\pi}}e^{-\frac{m y^2}{2}} \, dy, 
\end{equation}
\begin{displaymath}
\text{ and }\nu(dy) = \frac{1}{\sqrt{2\pi}} e^{-\frac{y^2}{2}}\, dy. 
\end{displaymath}
Note that $\mu_x$ is indeed a probability measure on $\rbb$ by Assumption~\ref{cond:Phi}.  
We denote by $\mu$ the product probability measure on $\rbb^\infty$ given by 
\begin{equation}\label{eqn:mu}
\mu=\mu_x\times\mu_v\times\prod_{k\geq 1}\nu.
\end{equation}
Observe that since $s>1/2$ by way of Assumption~\ref{cond:sol'n}
\begin{equation} \label{ineq:mu:second-moment}
\begin{aligned}
\int_{\rbb^\infty}\!\!\! \|X\|^2_{-s}\mu(dX) &=   \int_{\rbb^\infty}\!\!\! x^2+v^2+\sum_{k\geq 1}k^{-2s}z_k^2\ \mu(dX)< \infty.
\end{aligned}
\end{equation}
Thus the restriction of $\mu$ to $\H_{-s}$ is a probability measure as the above calculation shows that $\|X\|_{-s} < \infty$ $\mu$-almost surely. 

Let $X(t)=(x(t), v(t), z_1(t), z_2(t), \ldots)$ be the solution of \eqref{eqn:GLE-Markov} and define the operator $\PMarkov(t):\B_b\left(\H_{-s}\right)\to\B_b\left(\H_{-s}\right)$ by 
\begin{equation}\label{defn:Markov-semigroup}
\PMarkov(t)\varphi(X_0)=\E{\varphi(X(t,X_0))}.
\end{equation}
Here $\B_b(\H_{-s})$ is the space of bounded Borel measurable $\varphi: \H_{-s} \rightarrow \rbb$ and $C_b(\H_{-s})$ is the space of bounded continuous $\varphi: \H_{-s}\rightarrow \rbb$. It will be shown later in Proposition~\ref{prop:sol'n:feller} that $\left\{\PMarkov(t)\right\}_{t \geq 0}$ is a Feller Markov semigroup on $C_b(\H_{-s})$. We recall that a finite measure $\xi$ on $\H_{-s}$ is invariant for $\left\{\PMarkov(t)\right\}_{t \geq 0}$ if for every $\varphi\in C_b(\H_{-s})$ and $t>0$
\begin{displaymath}
\int_{\mathcal{H}_{-s}}\PMarkov(t)\varphi(X)\xi(dX)=\int_{\mathcal{H}_{-s}}\varphi(X)\xi(dX).
\end{displaymath}
We have the following result.   
\begin{theorem} \label{thm:density}
Suppose that Assumption~\ref{cond:Phi} and Assumption~\ref{cond:sol'n} are satisfied. Then $\mu$ defined by~\eqref{eqn:mu} is an invariant probability measure for the Markov semigroup $\left\{\PMarkov(t)\right\}_{t \geq 0}$ on $C_b(\H_{-s})$ defined by~\eqref{eqn:GLE-Markov}.  
\end{theorem} 
The idea behind the proof of Theorem~\ref{thm:density} is simple but the details are non-trivial.  This is because one tries to ``integrate by parts" in infinite-dimensions aiming to show that $\L^* \mu=0$ where $\L^*$ is some very formal adjoint of the Markov generator $\L$.  The way that we circumnavigate this is by showing $\mu$ is ``approximately invariant" for a sequence $X^R(t)$, $R\geq 1$, of processes which approximate $X(t)$ as $R\rightarrow \infty$.  This turns out to be enough to show that $\mu$ is invariant for the original process.

Finally, our last result concerns unique ergodicity in the diffusive regime.  To state it, we impose the following additional condition on the potential $\Phi(x)$.
\begin{assumption}\label{cond:Phi-1} There exist a function $f(x):\rbb\to\rbb^+$ that is bounded on bounded sets and a positive number $q$ such that for all $x,y\in\rbb$,
\begin{equation} \label{ineq:Phi-1}
\left| \Phi'(x)-\Phi'(y) \right|\leq  |x-y|\left(f(x-y)+\Phi(x)^q\right).
\end{equation}

\end{assumption}
\begin{remark} Assumption \ref{cond:Phi-1} is essentially a growth bound on the second derivative of $\Phi$. In particular, if we assume further that there exist $c,\, q_1,\,q_2>0$ such that for all $x,\, y\in\rbb$ and $t\in[0,1]$,
\begin{equation}\label{cond:Phi:second-deriv}
\begin{aligned} 
\left|\Phi''(x)\right|&\leq c(\Phi(x)^{q_1}+1),\\
\text{and}\quad\Phi((1-t)x+ty)&\leq c(\Phi(x)^{q_2}+\Phi(|x-y|)^{q_2}+1),
\end{aligned}
\end{equation}
then $\Phi$ satisfies Assumption \ref{cond:Phi-1}. From this observation, it is a short exercise to see that Condition~\ref{cond:Phi:second-deriv} includes not only the class of non-negative polynomials of even degree, but also functions that even grow faster than a polynomial at infinity, e.g. $\Phi(x) = e^{x^2}$. 
\end{remark}

\begin{theorem} \label{thm:inv-measure:unique} Suppose Assumption \ref{cond:Phi}, Assumption \ref{cond:Phi-1} and Condition (D) of Assumption~\ref{cond:sol'n} are satisfied. Then $\mu$ is the unique invariant measure for the Markov process defined by~\eqref{eqn:GLE-Markov}.  
\end{theorem}

The proof of Theorem~\ref{thm:inv-measure:unique}, which will be given in Section~\ref{sec:invariant-unique}, uses a asymptotic coupling argument following the ideas and results in the works of \cite{glatt2017unique, hairer2011asymptotic, kulik2015generalized}. For a similar rigorous study of finite-dimesional Langevin Equation, we refer the reader to \cite{mattingly2002ergodicity,ottobre2011asymptotic}.

\section{Well-posedness}\label{sec:sol'n}\label{sec:well-posed}
For notational convenience, throughout we write~\eqref{eqn:GLE-Markov} more compactly as the following semilinear stochastic evolution equation
\begin{equation} \label{eqn:infinite-mode-short}
d X(t) = \left(LX(t)+F(X(t)\right)d t+B \, d W(t),\qquad X(0)=X_0\in\H_{-s},
\end{equation}
where $L$ is a linear map given by
\begin{equation} \label{eqn:generator}
LX = \begin{pmatrix}
0 & 1 & 0 & 0  & \dots\\
0 & -\frac{\gamma}{m} & -\frac{\sqrt{c_1}}{m} &  -\frac{\sqrt{c_2}}{m} & \dots\\
0&\sqrt{c_1} &-\lambda_1 &0  & \dots\\
0&\sqrt{c_2} & 0  & -\lambda_2  &\dots\\
\vdots & \vdots &\vdots &\vdots &\ddots 
\end{pmatrix}\begin{pmatrix}
x\\v\\z_1\\z_2\\ \vdots
\end{pmatrix}=\begin{pmatrix}
v \\ -\frac{\gamma}{m} v-\frac{1}{m}\sum_{k\geq 1} \sqrt{c_k} z_k\\
-\lambda_1 z_1+ \sqrt{c_1}v\\
-\lambda_2 z_2+ \sqrt{c_2}v\\
\vdots
\end{pmatrix},
\end{equation}
where $c_k,\,\lambda_k$ are as in~\eqref{c-k} and the force $F$ is defined as
\begin{align*}
F(X)=(0, -\Phi'(x)/m,0,0, \dots)^T.
\end{align*}
Regarding the stochastic term $B\, dW$, we may formally write
\begin{align*}
 B\, dW(t)=\begin{pmatrix}
 0& 0&0 &\dots\\
 \frac{\sqrt{2\gamma}}{m}&0&0&\cdots\\
 0&\sqrt{2\lambda_1}&0&\cdots\\
 \vdots&\vdots&\vdots&\ddots
 \end{pmatrix}\begin{pmatrix}
 dW_0\\dW_1\\dW_2\\ \vdots
 \end{pmatrix}=\begin{pmatrix}
0\\ \frac{\sqrt{2\gamma}}{m}dW_0(t) \\
\sqrt{2\lambda_1}dW_1(t)\\ \sqrt{2\lambda_2}dW_2(t) \\ \vdots
\end{pmatrix}.
\end{align*}
Following the formulation in, for example, \cite{da2014stochastic}, fix an auxiliary Hilbert space $\W$ and pick a complete orthonormal basis $\{e_k^{\W}\}_{k\geq 0}$. The cylindrical Wiener process $W(t)$ on the Hilbert space $\W$ is then defined as
\begin{equation}\label{eqn:Wiener}
W(t)=W_0(t)e^{\W}_0+W_1(t)e^{\W}_1+W_2(t)e^{\W}_2+\dots,
\end{equation}
where the sequence $\{W_k(t)\}_{k\geq 0}$ are independent one-dimensional Brownian Motions. We can then define $B:\W\to\H_{-s}$ by its action
\begin{equation} \label{eqn:B}
Be^{\W}_0=\frac{\sqrt{2\gamma}}{m}e_v,\qquad\text{and}\qquad Be^{\W}_k=\sqrt{2\lambda_k}k^{-s}e_k,\ k\geq 1,
\end{equation}
where $\left\{e_x,e_v,e_1,e_2,\dots\right\}$ is the canonical basis of $\H_{-s}$, cf.~\eqref{form:H-basis}. In view of \eqref{eqn:B}, we have
\begin{equation} \label{eqn:density-4}
BB^*e_x=0,\qquad BB^*e_v=\frac{2\gamma}{m^2}e_v,\qquad\text{and}\qquad BB^*e_k=2\lambda_k k^{-2s}e_k,\ k\geq 1.
\end{equation}

In order to prove well-posedness of Equation~\eqref{eqn:infinite-mode-short}, we need the following basic fact.  
\begin{proposition} \label{prop:L:bounded} Suppose that $\alpha, \beta$ as in~\eqref{c-k} and $s$ as in \eqref{eqn:H_p} satisfy $0\leq 2s<\alpha\beta$. Then, $L:\H_{-s}\to \H_{-s}$ defined in \eqref{eqn:generator} is a bounded linear operator.
\end{proposition}
\begin{proof} Recalling~\eqref{c-k}, \eqref{eqn:H-norm}, \eqref{eqn:generator} and invoking Cauchy–Schwarz inequality, we estimate
\begin{equation*}
\begin{aligned}
\MoveEqLeft[0]\|LX\|_{-s}^2\\
& =  v^2+\Big(-\frac{\gamma}{m} v-\frac{1}{m}\sum_{k\geq 1} \sqrt{c_k} z_k\Big)^2+\sum_{k\geq 1}k^{-2s}(-\lambda_k z_k+\sqrt{c_k}v)^2\\
&\leq v^2+\frac{2\gamma^2}{m^2}v^2+\frac{2}{m^2}\Big(\sum_{k\geq 1} \sqrt{c_k} z_k\Big)^2+\sum_{k\geq 1}2\lambda_k^2k^{-2s}z_k^2+\sum_{k\geq 1}2k^{-2s}c_k v^2\\
&\leq \Big(1+\frac{2\gamma^2}{m^2}+2\sum_{k\geq 1}k^{-2s}c_k\Big)v^2+ \frac{2}{m^2}\sum_{k\geq 1}c_k k^{2s}\sum_{k\geq 1}k^{-2s} z_k^2+2\sum_{k\geq 1}\lambda_k^2 k^{-2s}z_k^2 \\
&=\Big(1+\frac{2\gamma^2}{m^2}+2\sum_{k\geq 1}k^{-2s}k^{-(1+\alpha\beta)}\Big)v^2+ \frac{2}{m^2}\sum_{k\geq 1}k^{-(1+\alpha\beta)} k^{2s}\sum_{k\geq 1}k^{-2s} z_k^2\\
&\qquad\qquad\qquad\qquad\qquad\qquad\qquad\qquad\!+2\sum_{k\geq 1}k^{-\beta} k^{-2s}z_k^2 \\
&\leq \Big(1+\frac{2\gamma^2}{m^2}+2\sum_{k\geq 1}k^{-(1+\alpha\beta+2s)}+\frac{2}{m^2}\sum_{k\geq 1}k^{-(1+\alpha\beta-2s)}+2\Big)\|X\|_{-s}^2,
\end{aligned}
\end{equation*}
Since $\sum_{k\geq 1}k^{-(1+\alpha\beta-2s)}$ converges whenever $\alpha\beta>2s$, the desired result follows.
\end{proof}

The proof of Proposition~\ref{prop:sol'n} follows a Lyapunov argument which we now explain.  Let $\theta_R\in C^\infty(\rbb ; [0,1])$  satisfy
\begin{align} \label{defn:theta-R}
\theta_R(x) = \begin{cases}
1 & \text{ if } |x| \leq R \\
0 & \text{ if } |x| \geq R+1
\end{cases}
\end{align}
and consider ``cutoff" equation corresponding to~\eqref{eqn:infinite-mode-short}
\begin{equation}\label{eqn:infinite-mode-short-cut}
d X^R(t) = \left[LX^R(t)+F\left(X^R(t)\right)\theta_R\left(x^R(t)\right)\right]d t+B \, d W(t),\qquad X^R(0)=X_0 .  
\end{equation}
For each $R>0$, it is not hard to prove that the global (in time) solution $X^R$ exists and is unique, giving local (up until the time of explosion) pathwise existence and uniqueness of~\eqref{eqn:infinite-mode-short}.  Then, using a Lyapunov function $\Psi(X)$ that dominates the norm of $X$ in $\H_{-s}$, we show a global bound on these solutions that does not depend on $R$, thereby obtaining global solutions of~\eqref{eqn:infinite-mode-short}.

To prove Proposition~\ref{prop:sol'n}, we begin with the following proposition.  
\begin{proposition}[Local Existence] \label{prop:local-sol'n} Suppose that $\alpha, \beta$ as in~\eqref{c-k} and $s$ as in \eqref{eqn:H_p} satisfy $1< 2s<\alpha\beta$. Let $X_0\in\H_{-s}$ be given. For each $R>0$, there exists a unique $X^R(t)\in L^2\left(\Omega,C\left([0,\infty);\H_{-s}\right)\right)$ satisfying \eqref{eqn:infinite-mode-short-cut}. 
\end{proposition}
\begin{remark} We see that the conditions stated in Assumption~\ref{cond:sol'n} meet the hypothesis $1<2s<\alpha\beta$ of Proposition~\ref{prop:local-sol'n}.  
\end{remark}
\begin{proof}[Proof of Proposition \ref{prop:local-sol'n}]
The linear map $L$ is bounded by Lemma \ref{prop:L:bounded} and the nonlinear term in~\eqref{eqn:infinite-mode-short-cut} is globally Lipschitz in $\| \cdot \|_{-s}$ by construction. Moreover, the additive noise term lives in $\H_{-s}$ almost surely since
\begin{align*}
\Enone \Big\|\int_0^T\!\! B\, dW(t)\Big\|^2_{-s}& =\Enone\Big|\int_0^T\!\!\frac{\sqrt{2\gamma}}{m}dW_0(t)\Big|^2 + \sum_{k\geq 1}k^{-2s}\Enone\Big|\int_0^T\!\! \sqrt{2\lambda_k} \, dW_k(t)\Big|^2\\
&=\frac{2\gamma T}{m^2}+2T\sum_{k\geq 1}k^{-2s}\lambda_k<\infty.
\end{align*}
The corresponding solution hence exists and is unique by a standard Banach fixed point argument. 
\end{proof}
Next, inspired by \cite{mattingly2002ergodicity,pavliotis2014stochastic}, we introduce a Lyapunov function $\Psi:\H_{-s}\to[0,\infty)$ given by
\begin{equation} \label{eqn:Lyapunov}
\Psi(X):=\frac{1}{m}\Phi(x)+\frac{1}{2} v^2+\frac{1}{2}\sum_{k \geq 1}k^{-2s}z_k^2.
\end{equation}
Define $\mathcal{L}:C^2(\H_{-s})\to\rbb$ to be the operator given by
\begin{equation*}
\L\f(X):= \la D\f(X),LX+F(X)\ra+\frac{1}{2}\Tr( D^2\f BB^*).
\end{equation*}
In view of~\eqref{eqn:Frechet:first}, \eqref{eqn:Frechet:second}, \eqref{eqn:generator} and \eqref{eqn:density-4}, $\L$ can be explicitly written as
 \begin{multline} \label{defn:Markov-gen}
\L\f(X) = v\frac{\partial\f(X)}{\partial x}+\Big(\!\!-\frac{\gamma}{m}v-\frac{1}{m}\Phi'(x)-\frac{1}{m}\sum_{k\geq 1}\sqrt{c_k}z_k\Big)\frac{\partial\f(X)}{\partial v}\\
+\sum_{k\geq 1} \left(-\lambda_kz_k+\sqrt{c_k}v\right)\frac{\partial\f(X)}{\partial z_k}+ \frac{\gamma}{m^2}\frac{\partial ^2\f(X)}{\partial v^2}+\sum_{k\geq 1}\lambda_k\frac{\partial ^2\f(X)}{\partial z_k^2}
\end{multline}
where $\varphi\in C^2(\H_{-s})$. Note that, once we establish Proposition~\ref{prop:sol'n}, $\L$ is in fact the infinitesimal generator of the Markov semigroup $\left\{\PMarkov(t)\right\}_{t \geq 0}$ associated with~\eqref{eqn:GLE-Markov}. We assert the following proposition.

\begin{proposition}[Global bound] \label{prop:Ito-Lyapunov} Suppose that Assumption \ref{cond:Phi} and Assumption~\ref{cond:sol'n} are satisfied. Let $\Psi(X)$ be defined in \eqref{eqn:Lyapunov} and $\L$ be the operator as in \eqref{defn:Markov-gen}. Then, for every $X\in\H_{-s}$,
\begin{equation} \label{ineq:Ito-Lyapunov}
\L \Psi(X) \leq a_1\Psi(X)+a_2,
\end{equation}
where $a_1,\, a_2$ are finite constants that can be explicitly given as
\begin{displaymath}
a_1=\max\Big\{1+\frac{1}{m},\frac{1}{m}\sum_{k\geq 1}c_kk^{2s}+\sum_{k\geq 1}c_kk^{-2s}\Big\}\text{  and  }a_2=\frac{\gamma}{m^2}+\sum_{k\geq 1}k^{-2s}\lambda_k.
\end{displaymath}
\end{proposition}
\begin{proof} Applying Assumption~\ref{cond:sol'n}, first note that $a_1$ and $a_2$ are both finite since 
\begin{displaymath}
\sum_{k\geq 1}c_kk^{2s}=\sum_{k\geq 1}\frac{1}{ k^{1+\alpha\beta-2s}}<\infty\quad\text{and}\quad \sum_{k\geq 1}k^{-2s}\lambda_k=\sum_{k\geq 1}\frac{1}{k^{\beta+2s}}<\infty.  
\end{displaymath}
We now apply $\L$ to $\Psi$ to see that
\begin{equation}\label{eqn:Ito-Lyapunov}
\begin{aligned}
\L\Psi(X)
&=-\frac{\gamma}{m}v^2-\sum_{k\geq 1}\lambda_kk^{-2s}z_k^2\\
& \qquad -\frac{1}{m}\sum_{k\geq 1}\sqrt{c_k} z_kv+\sum_{k\geq 1}\sqrt{c_k}k^{-2s}z_kv+\frac{\gamma}{m^2}+\sum_{k\geq 1}k^{-2s}\lambda_k .  
\end{aligned}
\end{equation}
The cross terms between $z_k$ and $v$ can be bounded using H\"{o}lder's inequality as follows: 
\begin{equation}\label{eqn:Ito-Lyapunov:1}
\begin{aligned}   
\MoveEqLeft[2] -\frac{1}{m}\sum_{k\geq 1}\sqrt{c_k} z_k v +\sum_{k\geq 1}\sqrt{c_k}k^{-2s}z_k v  \\
&   \leq \frac{1}{2m}\Big(\sum_{k\geq 1}c_kk^{2s}\sum_{k\geq 1}k^{-2s}z_k ^2+v ^2\Big)+\frac{1}{2}\Big(\sum_{k\geq 1}c_kk^{-2s}\sum_{k\geq 1}k^{-2s}z_k ^2+v ^2\Big)\\
 &= \frac{1}{2}\Big(1+\frac{1}{m}\Big)v ^2+\frac{1}{2}\Big(\frac{1}{m}\sum_{k\geq 1}c_kk^{2s}+\sum_{k\geq 1}c_kk^{-2s}\Big)\sum_{k\geq 1}k^{-2s}z_k ^2\\
 &\leq a_1 \Psi(X),
\end{aligned}
\end{equation}
where $a_1>0$ is as in the statement of the result. We finally combine \eqref{eqn:Ito-Lyapunov} with \eqref{eqn:Ito-Lyapunov:1} to obtain \eqref{ineq:Ito-Lyapunov}.
\end{proof}

We are now ready to prove the main existence and uniqueness result for equation~\eqref{eqn:infinite-mode-short}. The argument is classical and can be found in literature, e.g.,  \cite{albeverio2008spde,glatt2009strong,jacod2006calcul}.

\begin{proof}[Proof of Proposition \ref{prop:sol'n}] For every $R>0$, let $X^R(t)$ be the unique solution of the cutoff system \eqref{eqn:infinite-mode-short-cut} given to us by Proposition \ref{prop:local-sol'n}. Define the stopping time 
\[\tau_R=\inf\left\{t>0: \|X(t)\|_{-s}>R\right\}.\]
Note that, for all times $t< \tau_R$, $X^R$ solves \eqref{eqn:infinite-mode-short}. Consequently, the solution \eqref{eqn:infinite-mode-short} exists and is unique up until the \emph{time of explosion} $\tau_\infty=\lim_{R\to\infty}\tau_R$, which is possibly finite on a set of positive probability. We show using the Lyapunov function $\Psi$ above, cf.~\eqref{eqn:Lyapunov} that $\tau_\infty=\infty$ a.s.

By Ito's Formula we have that
\begin{equation}\label{eqn:Ito-Lyapunov:2}
\begin{aligned}
 d\Psi(X(t\mi\tau_R))=\mathcal{L}\Psi(X(t\mi\tau_R))dt & + \frac{\sqrt{2\gamma}}{m}v(t\mi\tau_R)dW_0(t)\\
 &+ \sum_{k\geq 1}\sqrt{2\lambda_k}k^{-2s}z_k(t\mi\tau_R)dW_k(t),
\end{aligned}
\end{equation}
where $\L$ is the operator defined in \eqref{defn:Markov-gen}. We then infer the following bound\begin{multline}\label{ineq:Ito-Lyapunov:2a}
\Enone\Big[\sup_{0\leq t\leq T} \Psi(X(t\wedge\tau_R)) \Big]
\leq \Psi(X_0)+\Enone\sup_{0\leq t\leq T} \int_0^t\!\!\L\Psi(X(r\wedge\tau_R))dr\\+\Enone\sup_{0\leq t\leq T}\Big|\int_0^t\!\! \frac{\sqrt{2\gamma}}{m}v(r\mi\tau_R)dW_0(r)+ \sum_{k\geq 1}\sqrt{2\lambda_k}k^{-2s}z_k(t\mi\tau_R)dW_k(r)dr \Big|.
\end{multline}
Applying Proposition \ref{prop:Ito-Lyapunov} on $\L\Psi(X(r\mi\tau_R))$ and the Burkholder-Davis-Gundy inequality on the martingale term on the above RHS yields
\begin{multline}\label{ineq:Ito-Lyapunov:3}
\Enone \Big[\sup_{0\leq t\leq T} \Psi(X(t\wedge\tau_R)) \Big]
\leq \Psi(X_0)+a_2T+a_1\Enone\int_0^T\!\sup_{0\leq r\leq t} \Psi(X(r\wedge\tau_R))dt\\+c\Big[\Enone\int_0^T\! \frac{2\gamma}{m^2}v(t\wedge\tau_R)^2
+\sum_{k\geq 1}2\lambda_k k^{-4s}z_k(t\wedge\tau_R)^2dt \Big]^{1/2},
\end{multline}
where $a_1,\, a_2$ are as in \eqref{ineq:Ito-Lyapunov} and $c>0$ is the constant from Burkholder-Davis-Gundy's inequality and is independent of $R,\, T,\, X_0$. We observe now that the last integrand in \eqref{ineq:Ito-Lyapunov:3} is dominated by $\Psi(X(t\wedge\tau_R))$. We thus infer that
\begin{equation}\label{ineq:Ito-Lyapunov:4}
\Enone \Big[\sup_{0\leq t\leq T} \Psi(X(t\wedge\tau_R))\Big]\leq \Psi(X_0)+c_1+ c_2\int_0^T\Enone\Big[\sup_{0\leq r\leq t} \Psi(X(r\wedge\tau_R))\Big]dt,
\end{equation}
where $c_1,\,c_2>0$ are constants independent of $R$ and $X_0$. Gronwall's inequality then implies that
\begin{equation}\label{ineq:Ito-Lyapunov:5}
\Enone \Big[\sup_{0\leq t\leq T} \Psi(X(t\wedge\tau_R)) \Big]\leq \left(\Psi(X_0)+c_1\right)e^{c_2 T}.
\end{equation}
Also note that there exists a constant $c>0$ such that for all $X\in\H_{-s} $, $c(\Psi(X)+1)\geq \|X\|^2_{-s}$. We thus infer the existence of a constant $C(T,X_0)>0$ such that
\begin{equation}\label{ineq:Ito-Lyapunov:6}
\Enone\Big[\sup_{0\leq t\leq T} \|X(t\wedge\tau_R)\|_{-s}^2 \Big]\leq C(T,X_0).
\end{equation}
Sending $R$ to infinity in the above, we obtain by Fatou's lemma  
\begin{equation}\label{ineq:Ito-Lyapunov:7}
\Enone\Big[\sup_{0\leq t\leq T} \|X(t\wedge\tau_\infty)\|_{-s}^2 \Big]\leq C(T,X_0).
\end{equation}
Hence, this implies $\P{T<\tau_\infty}=1$ for any $T>0$. By sending $T$ to infinity, we see that $\P{\tau_\infty=\infty}=1$, which implies the well-posedness of the global solution for any fixed $X_0 \in \mathcal{H}_{-s}$.  

We finally note that $C(T,X_0)$ is actually dominated by $\Psi(X_0)e^{c_2T}$ following from estimates~\eqref{ineq:Ito-Lyapunov:5} and~\eqref{ineq:Ito-Lyapunov:6}. It is also clear that $\Psi(X)$ is bounded on bounded sets in $\Hs$. We therefore obtain the bound in~\eqref{ineq:strong-sol'n:sup}, which concludes the proof.
\end{proof}
In addition to pathwise existence and uniqueness of the solution of equation~\eqref{eqn:infinite-mode-short}, we will need the following basic properties of the Markov semigroup $\PMarkov(t):\B_b(\H_{-s})\to\B_b(\H_{-s})$. We recall that $\PMarkov(t)$  defined as in \eqref{defn:Markov-semigroup} possesses the Markov property; namely, for every $X\in \H_{-s}$, $\f\in C_b(\H_{-s})$, $t,\, r\geq 0$,
\[\PMarkov(t+r)\f(X)=\PMarkov(t)\left(\PMarkov(r)\f\right)(X).\]
\begin{proposition}\label{prop:sol'n:feller} Under the Hypothesis of Proposition \ref{prop:sol'n}, let $X(t)$ be the unique strong solution of \eqref{eqn:GLE-Markov} and $\PMarkov(t)$ be the corresponding Markov semigroup. We have the following:
\begin{itemize}
\item[(a)] Whenever $X_k\rightarrow X_0$ in $\H_{-s}$ 
\begin{equation}\label{eqn:sol'n-1}
\lim_{k\to\infty}\Enone\|X(t,X_k)-X(t,X_0)\|_{-s}=0.
\end{equation}
\item[(b)] $\PMarkov(t)$ has the Feller property: $\PMarkov(t)\varphi\in C_b(\H_{-s})$ whenever $\varphi\in C_b(\H_{-s})$.
\end{itemize}
\end{proposition}  
\begin{proof} (a) For notational simplicity, throughout this proof, we shall omit the subscript $-s$ in the norm $\|\cdot\|_{-s}$. Denote by $X_{(k)}(t)$ the solution of \eqref{eqn:GLE-Markov} with initial data $X_{(k)}(0)=X_k$. Fixing $R>0$ to be chosen later, define the stopping time 
\begin{equation*}
\tau_R^k = \inf_{t\geq 0}\{\|X_{(k)}(t)\|+\|X_{(0)}(t)\|\geq R\},
\end{equation*}
and observe that by Chebychev's inequality,
\begin{align*}
\P{\tau_R^k\leq t}&=\Pnone\Big\{\sup_{0\leq \ell\leq t}\|X_{(k)}(\ell)\|+\|X_{(0)}(\ell)\|\geq R\Big\}\\
&\leq \frac{\E{\sup_{0\leq \ell\leq t} \|X_{(k)}(\ell)\|+\|X_{(0)}(\ell) }}{R}\\
&\leq \frac{\sqrt{C(t,X_k)}+\sqrt{C(t,X_0)}}{R},
\end{align*}
where we used~\eqref{ineq:strong-sol'n} in the last inequality. It follows from~\eqref{ineq:strong-sol'n:sup} that 
\begin{equation}\label{ineq:sol'n-2a}
\sup_{k}\P{\tau_R^k\leq t}\leq \frac{C(t)}{R},
\end{equation}
for a finite constant $C(t)>0$ independent of $R$. Next, let $X_{(k)}^R$ and $X_{(0)}^R$ be the local solutions of \eqref{eqn:infinite-mode-short-cut} from Proposition \ref{prop:local-sol'n}. Since the drift term of \eqref{eqn:infinite-mode-short-cut} is Lipschitz, there exists a constant $c(R,t)>0$ such that (see Theorem 9.1, \cite{da2014stochastic}).
\begin{equation}\label{ineq:sol'n-2}
\Enone\big\|X_{(k)}^R(t)-X_{(0)}^R(t)\big\|\leq C(R,t)\left\|X_k-X_0\right\|.
\end{equation}
Now we have a chain of implications
\begin{equation*}
\begin{aligned}
\MoveEqLeft[1]\Enone\|X_{(k)}(t)-X_{(0)}(t)\|\\
&\leq \Enone \Big[\big(\|X_{(k)}(t)\|+\|X_{(0)}(t)\|\big)1_{\{\tau_R^k\leq t\}}\Big]+\E{\|X_{(k)}(t)-X_{(0)}(t)\|1_{\{\tau_R^k> t\}}}\\
&=  \E{\big(\|X_{(k)}(t)\|+\|X_{(0)}(t)\|\big)1_{\{\tau_R^k\leq t\}}}+\E{\|X_{(k)}^R(t)-X_{(0)}^R(t)\|1_{\{\tau_R^k> t\}}}\\
&\leq \Big(\E{\big(\|X_{(k)}(t)\|+\|X_{(0)}(t)\|\big)^2 }\Big)^{1/2}\Big(\P{\tau_R^k\leq t}\Big)^{1/2}+ \Enone\|X_{(k)}^R(t)-X_{(0)}^R(t)\|\\
&\leq \Big(\E{\big(\|X_{(k)}(t)\|+\|X_{(0)}(t)\|\big)^2 }\Big)^{1/2}\Big(\P{\tau_R^k\leq t}\Big)^{1/2}+ C(R,t)\left\|X_k-X_0\right\|\\
&\leq \frac{C_1(t)}{R^{1/2}}+C(R,t)\left\|X_k-X_0\right\|,
\end{aligned}
\end{equation*}
where note carefully that $C_1(t)$ is independent of $k$ and $R$ and $C(R,t)$ is independent of $k$. The above RHS now tends to zero by taking $R$ sufficiently large first and then letting $X_k$ sufficiently close to $X_0$. This establishes (a).

To prove (b), let $X_k\to X_0$ and $\varphi\in C_b\left(\H_{-s}\right)$, we have to show $\PMarkov(t)\varphi(X_k)\to\PMarkov(t)\varphi(X_0)$. It suffices to show that for every subsequence $\{k_i\}$, there exists a further subsequence $\{k_{i_j}\}$ such that $\PMarkov(t)\varphi(X_{k_{i_j}})\to\PMarkov(t)\varphi(X_0)$. In view of part (a), the sequence $X_{(k_i)}(t)$ converges to $X_{(0)}(t)$ in $L^1(\Omega;\Hs)$. We thus can extract a subsequence $X_{(k_{i_j})}(t)$ converging to $X_{(0)}(t)$ a.s. Since $\varphi$ is bounded, applying the Dominated Convergence Theorem yields
\begin{align*}
\Enone \big[\varphi(X_{(k_{i_j})}(t))\big]\to \E{\varphi(X_{(0)}(t))} \text{ as } j\rightarrow \infty, 
\end{align*}
which implies (b) and thus completes the proof.
\end{proof}

\section{Invariance of $\mu$} \label{sec:invariant-exist}
In this section, we show that $\mu$ defined in \eqref{eqn:mu} is invariant for the Markov semigroup $\left\{\PMarkov(t)\right\}_{t\geq 0}$.  We first sketch briefly the structure of the proof before diving into details. 

The goal is to show that for every $\varphi\in C_b(\H_{-s})$ and $t\geq 0$ we have
\begin{equation}\label{thm:density-1}
\int_{\mathcal{H}_{-s}}\!\!\!\!\PMarkov(t)\varphi(X)\mu(dX)=\int_{\mathcal{H}_{-s}}\!\!\!\!\varphi(X)\mu(dX).
\end{equation}
Let $C_b^2(\H_{-s})$ denote the space of real-valued functions on $\Hs$ that have bounded first and second Fr\'echet derivatives. Approximating $\varphi$ by functions in $C^2_b(\Hs)$ if necessary, it thus suffices to show that \eqref{thm:density-1} holds for any $\psi\in C_b^2(\H_{-s})$
 \begin{equation}\label{thm:density-1a}
\int_{\mathcal{H}_{-s}}\!\!\!\!\PMarkov(t)\psi(X)\mu(dX)=\int_{\mathcal{H}_{-s}}\!\!\!\!\psi(X)\mu(dX).
\end{equation}
In order to show~\eqref{thm:density-1a}, it is helpful to make use of the cutoff system~\eqref{eqn:infinite-mode-short-cut} and the semigroup $\PMarkov^R(t)$ where for $R>0$, $\PMarkov^R(t)$ is defined analogously to~\eqref{defn:Markov-semigroup} by replacing $X(t)$ with $X^R(t)$ solving~\eqref{eqn:infinite-mode-short-cut}. The advantage of using the cutoff systems is that, because they have globally Lipschitz coefficients, they immediately satisfy the Kolmogorov backward equation, cf. Theorem 9.23 of~\cite{da2014stochastic}.  This fact we will need later in the proof of Proposition~\ref{prop:density-cut}. Specifically, we will prove that $\mu$ is \emph{almost invariant} for the cutoff semigroup $\PMarkov^R(t)$; namely,
\begin{equation}\label{thm:density-2}
\int_{\mathcal{H}_{-s}}\!\!\!\!\PMarkov^R(t)\psi(X)\mu(dX)=\int_{\mathcal{H}_{-s}}\!\!\!\!\psi(X)\mu(dX)+\varepsilon^R(\psi,t),
\end{equation} 
where $\varepsilon^R(\psi,t)$ is a remainder term that (possibly) depends on $\psi$ and $t$, and satisfies $\varepsilon^R(\psi,t)\to 0$ as $R\to \infty$. We will see that this then implies the desired equality~\eqref{thm:density-1a}.

Before proving Theorem \ref{thm:density}, we collect several properties about Gaussian measures on $\rbb$ which follow simply by using integration by parts. Let $\mu_v,\, \nu$ be as in~\eqref{defn:mu-x}. Then, for every  $\varrho_1\in C^2_b(\rbb)$, it holds that
\begin{equation}\label{eqn:Gaussian-0}
\int_\rbb -y\varrho_1'(y)+\frac{1}{m}\varrho_1''(y)\, \mu_v(dy)=0,
\end{equation}
and
\begin{equation}\label{eqn:Gaussian-1}
\int_\rbb -y\varrho_1'(y)+\varrho_1''(y)\, \nu(dy)=0.
\end{equation}
Also, for every $\varrho_2\in C^1_b(\rbb^2)$, we have
\begin{equation}\label{eqn:Gaussian-3}
\int_{\rbb^2}\Big( \frac{1}{m}z\partial_y\varrho_2(y,z)-y\partial_z\varrho_2(y,z)\Big)\left(\mu_v\times\nu\right)(dy,dz)=0.
\end{equation}

With these observations, we have the following result:
\begin{lemma} \label{lem:gen-cut} Given $R>0$, let $\L^R$ be the infinitesimal generator of the Markov semigroup $\{\PMarkov^R(t)\}_{t\geq 0}$ associated with $X^R$ solving~\eqref{eqn:infinite-mode-short-cut} and let $\mu$ be as in~\eqref{eqn:mu}. Then, for every $\psi\in C^2_b(\H_{-s})$, we have the following equality 
\begin{equation} \label{eqn:gen-cut-1}
\int_{\mathcal{H}_{-s}}\!\!\!\! \L^R\psi(X)\mu(dX) = \int_{\mathcal{H}_{-s}}\!\! \!\!v\Phi'(x)\left(1-\theta^R(x)\right)\psi(X)\mu(dX).
\end{equation}
\end{lemma}
\begin{proof} Similar to \eqref{defn:Markov-gen}, $\L^R\psi$ is given by
\begin{multline} \label{eqn:density-7}
\L^R\psi(X) = v\frac{\partial\psi(X)}{\partial x}+\Big(\!\!-\frac{\gamma}{m}v-\frac{1}{m}\Phi'(x)\theta^R(x)-\frac{1}{m}\sum_{k\geq 1}\sqrt{c_k}z_k\Big)\frac{\partial\psi(X)}{\partial v}\\
+\sum_{k\geq 1} \left(-\lambda_kz_k+\sqrt{c_k}v\right)\frac{\partial\psi(X)}{\partial z_k}+ \frac{\gamma}{m^2}\frac{\partial ^2\psi(X)}{\partial v^2}+\sum_{k\geq 1}\lambda_k\frac{\partial ^2\psi(X)}{\partial z_k^2}.
\end{multline}
We integrate both sides against $\mu$ in $\H_{-s}$   and rearrange the above RHS appropriately to obtain
\begin{equation} \label{eqn:gen-cut-7a}
\begin{aligned}
\int_{\mathcal{H}_{-s}} \L^R\psi(X)\mu(dx)
&= \int_{\mathcal{H}_{-s}} \left[v\frac{\partial \psi(X)}{\partial x}-\frac{1}{m}\Phi'(x)\theta^R(x)\frac{\partial \psi(X)}{\partial v}\right]\mu(dX)\\
&\qquad+\int_{\mathcal{H}_{-s}} \left[-\frac{\gamma}{m}v\frac{\partial \psi(X)}{\partial v}+\frac{\gamma}{m^2}\frac{\partial ^2\psi(X)}{\partial v^2}\right]\mu(dX)\\
&\qquad+\sum_{k\geq 1}\int_{\mathcal{H}_{-s}} \left[-\frac{\sqrt{c_k}}{m}z_k\frac{\partial \psi(X)}{\partial v}+\sqrt{c_k}v\frac{\partial \psi(X)}{\partial z_k}\right]\mu(dX)\\
&\qquad+\sum_{k\geq 1}\int_{\mathcal{H}_{-s}} \left[-\lambda_kz_k\frac{\partial \psi(X)}{\partial z_k}+\lambda_k\frac{\partial ^2\psi(X)}{\partial z_k^2}\right]\mu(dX)\\
&=I_{0,1}+I_{0,2}+\sum_{k\geq 1}I_{k,1}+\sum_{k\geq 1}I_{k,2}.
\end{aligned}
\end{equation}
At this point~\eqref{eqn:gen-cut-7a} is still formal. We need to show that $\L^R\psi\in L^1(\Hs,\mu)$ and that the above rearrangement is possible. To this end, we claim that the RHS after the first equality of \eqref{eqn:gen-cut-7a} is absolutely convergent. Since $\psi\in C^2_b(\H_{-s})$, \eqref{eqn:Frechet:first} and Parseval's identity  imply a bound on first partial derivatives
\begin{equation}\label{ineq:psi-bound-1st}
\begin{aligned}
\MoveEqLeft[5] \left|\frac{\partial \psi(X)}{\partial x}\right|^2+\left|\frac{\partial \psi(X)}{\partial v}\right|^2+\sum_{k\geq 1}k^{2s}\left|\frac{\partial \psi(X)}{\partial z_k}\right|^2 \\
&=\la D\psi(X),e_x\ra ^2+\la D\psi(X),e_v\ra^2+\sum_{k\geq 1}\la D\psi(X),e_k\ra^2\\
&=\| D\psi(X)\|_{-s}^2\\
&\leq\| D\psi\|_\infty^2,
\end{aligned}
\end{equation} 
where for $ D\psi:\Hs\to\Hs$, $\| D\psi\|_\infty=\sup_{Y\in\Hs}\| D\psi(Y)\|_{-s}$. Similarly, \eqref{eqn:Frechet:second} implies bounds on second partial derivatives,
\begin{equation}\label{ineq:psi-bound-second-xv}
\frac{\partial^2\psi\left(X \right)}{\partial x^2}=\la D^2\psi(X)(e_x),e_x\ra\leq \| D^2\psi\|_{\infty},\qquad \frac{\partial^2\psi\left(X \right)}{\partial v^2}\leq \| D^2\psi\|_{\infty},
\end{equation}
and
\begin{equation}\label{ineq:psi-bound-second-z}
k^{2s}\frac{\partial^2\psi\left(X \right)}{\partial z_k^2}\leq \| D^2\psi\|_{\infty},
\end{equation}
where for $ D^2\psi(X):\Hs\to L(\Hs,\Hs)$, 
\begin{displaymath}
\| D^2\psi\|_{\infty}=\sup_{\Hs}\| D^2\psi(Y)\|_{L(\Hs,\Hs)}.
\end{displaymath} 
In the RHS after the first equality of \eqref{eqn:gen-cut-7a}, the first four terms are bounded by, using \eqref{ineq:psi-bound-1st}, \eqref{ineq:psi-bound-second-xv},
\begin{multline}\label{ineq:gen-cut-8a}
\int_{\mathcal{H}_{-s}} \left| v\frac{\partial \psi(X)}{\partial x}\right|+\frac{1}{m}\left|\Phi'(x)\theta^R(x)\frac{\partial \psi(X)}{\partial v}\right|\\+ \frac{\gamma}{m}\left|v\frac{\partial \psi(X)}{\partial v}\right|+\frac{\gamma}{m^2}\left|\frac{\partial ^2\psi(X)}{\partial v^2}\right|\mu(dX)
\end{multline}
\begin{displaymath}
\leq\| D\psi\|_\infty\Big[\left(1+\frac{\gamma}{m}\right)\int_\rbb |v|\mu_v(dv)+\frac{1}{m}\int_\rbb \left|\Phi'(x)\right|\theta_R(x)e^{-\Phi(x)}dx\Big]+\frac{\gamma}{m^2}\| D^2\psi\|_\infty,
\end{displaymath}
which is finite, by the definition of $\mu_v$ from~\eqref{defn:mu-x} and the fact that $\theta_r$  as in~\eqref{defn:theta-R} has compact support. For the first sum on the third line of \eqref{eqn:gen-cut-7a}, we estimate as follows.
\begin{equation}\label{ineq:gen-cut-8b}
\begin{aligned}
\MoveEqLeft[3]\sum_{k\geq 1}\int_{\mathcal{H}_{-s}} \left|\frac{\sqrt{c_k}}{m}z_k\frac{\partial \psi(X)}{\partial v}\right|\mu(dX)\\
&\leq \frac{\| D\psi\|_\infty}{m}\Big(\sum_{k\geq 1}c_kk^{2s}\Big)^{1/2}\int_{\mathcal{H}_{-s}}\Big(\sum_{k\geq 1}k^{-2s}z_k^2\Big)^{1/2}\mu(dX)\\
&\leq \frac{\| D\psi\|_\infty}{m}\Big(\sum_{k\geq 1}c_kk^{2s}\Big)^{1/2}\int_{\mathcal{H}_{-s}}\!\!\!\!\|X\|_{-s}\mu(dX)<\infty,
\end{aligned}
\end{equation}
since by Assumption \ref{cond:sol'n}, $\sum_{k\geq 1}c_kk^{2s}$ is finite and so is $\int_{\mathcal{H}_{-s}}\|X\|_{-s}\mu(dX)$, by the definition of $\mu$, see~\eqref{ineq:mu:second-moment}. Similarly, for the second sum on the third line of \eqref{eqn:gen-cut-7a}, using \eqref{ineq:psi-bound-1st} again, we infer
\begin{equation}\label{ineq:gen-cut-8c}
\begin{aligned}
\MoveEqLeft[3]
\sum_{k\geq 1}\int_{\mathcal{H}_{-s}}\left|\sqrt{c_k}v\frac{\partial \psi(X)}{\partial z_k}\right|\mu(dX)\\
&\leq \int_{\mathcal{H}_{-s}} \Big(\sum_{k\geq 1}\frac{c_k}{k^{2s}}\Big)^{1/2}\Big(\sum_{k\geq 1}k^{2s}\Big|\frac{\partial \psi(X)}{\partial z_k}\Big|^2\Big)^{1/2}|v|\mu(dX)\\
&\leq  \| D\psi\|_\infty\Big(\sum_{k\geq 1}\frac{c_k}{k^{2s}}\Big)^{1/2}\int_{\mathcal{H}_{-s}}\!\!\!\!|v|\mu(dX)<\infty.
\end{aligned}
\end{equation}
For the first sum on the fourth line of \eqref{eqn:gen-cut-7a}, similar to \eqref{ineq:gen-cut-8c}, we invoke \eqref{ineq:psi-bound-1st} again to see that
\begin{equation} \label{ineq:gen-cut-8d}
\sum_{k\geq 1}\int_{\mathcal{H}_{-s}} \left|\lambda_kz_k\frac{\partial \psi(X)}{\partial z_k}\right|\mu(dX)\leq \| D\psi\|_\infty\int_{\mathcal{H}_{-s}}\!\!\!\!\|X\|_{-s}\mu(dX)<\infty.
\end{equation}
Lastly, we employ \eqref{ineq:psi-bound-second-z} to estimate the latter sum on the fourth line of \eqref{eqn:gen-cut-7a},
\begin{equation}\label{ineq:gen-cut-8e}
\begin{aligned}
\sum_{k\geq 1}\int_{\mathcal{H}_{-s}} \Big|\lambda_k\frac{\partial ^2\psi(X)}{\partial z_k^2}\Big|\mu(dX) &= \int_{\mathcal{H}_{-s}} \Big|\sum_{k\geq 1}\lambda_kk^{-2s}k^{2s}\frac{\partial ^2\psi(X)}{\partial z_k^2}\Big|\mu(dX)\\
&\leq \| D^2\psi\|_\infty \int_{\mathcal{H}_{-s}} \sum_{k\geq 1}\lambda_kk^{-2s}\mu(dX)<\infty.
\end{aligned}
\end{equation}
We can now apply Fubini Theorem on the Hilbert space $\H_{-s}$, see e.g. \cite{kuo2011integration}. For $X\in\H_{-s}$, we write  
\begin{equation*}
X=P_{x,v}X+P_{x,v}^\perp X=xe_x+ve_v+Z,
\end{equation*} 
where $P_{x,v}X=xe_x+ve_v$ is the projection on the subspace $\la\{e_x,e_v\}\ra$ and $P_{x,v}^\perp X=Z$ is the projection on $\la \{e_x,e_v\}\ra^\perp$. Then, $\mu$ can be decomposed as $\mu=\mu_{x,v}\times\mu_{x,v}^\perp$, where $\mu_{x,v}=\mu_x\times\mu_v$ is a measure on $P_{x,v}\H_{-s}$ and $\mu_{x,v}^\perp=\prod_{k\geq 1}\nu$ is a measure on $P_{x,v}^\perp\H_{-s}$. It follows that
\begin{equation}\label{eqn:density-8}
\begin{aligned}
\MoveEqLeft[3]
\int_{\mathcal{H}_{-s}}\!\!\!\!I_{0,1}(X)\mu(dX)\\
& = \int_{\mathcal{H}_{-s}}\!\!\!\!v\frac{\partial\psi(X)}{\partial x}-\frac{1}{m}\Phi'(x)\theta^R(x)\frac{\partial\psi(X)}{\partial v}\mu(dX)\\
&= \int_{P_{x,v}^\perp\H_{-s}}\int_{\rbb^2}v\frac{\partial\psi(X)}{\partial x}-\frac{1}{m}\Phi'(x)\theta^R(x)\frac{\partial\psi(X)}{\partial v}\mu_{x,v}(dx,dv)\mu_{ x,v}^\perp(dZ),
\end{aligned}
\end{equation}
where we use Fubini's Theorem in the last implication. This is possible since we already established the absolute convergence in \eqref{ineq:gen-cut-8a}. Integrating by parts the first integral against $\mu_{x,v}$ yields
\begin{equation}\label{eqn:density-8a}
\begin{aligned}
\int_{\mathcal{H}_{-s}}\!\!\!\!I_{0,1}(X)\mu(dX)
& =  \int_{P_{x,v}^\perp\H_{-s}}\int_{\rbb^2}v\Phi'(x)\left(1-\theta^R(x)\right)\psi(X)\mu_{x,v}(dx,dv)\mu_{ x,v}^\perp(dX)\\
&=\int_{\mathcal{H}_{-s}}\!\!\!\!v\Phi'(x)\left(1-\theta^R(x)\right)\psi(X)\mu(dX).
\end{aligned}
\end{equation}
Similarly for $I_{0,2}$, we have
\begin{equation}\label{eqn:density-8b}
\begin{aligned}
\MoveEqLeft[3]
\int_{\mathcal{H}_{-s}}\!\!\!\!I_{0,2}(X)\mu(dX)\\
& = \int_{\mathcal{H}_{-s}} \Big[-\frac{\gamma}{m}v\frac{\partial \psi(X)}{\partial v}+\frac{\gamma}{m^2}\frac{\partial ^2\psi(X)}{\partial v^2}\Big]\mu(dX)\\
&= \int_{P_{v}^\perp\H_{-s}}\int_{\rbb} \Big[-\frac{\gamma}{m}v\frac{\partial \psi(X)}{\partial v}+\frac{\gamma}{m^2}\frac{\partial ^2\psi(X)}{\partial v^2}\Big]\mu_{v}(dv)\mu_{v}^\perp(d(xe_x+Z))\\
&=0,
\end{aligned}
\end{equation}
where we have employed~\eqref{eqn:Gaussian-0} in the last implication. For the last two terms $I_{k,1}$, $I_{k,2}$, after integration by parts, we invoke \eqref{eqn:Gaussian-3}, \eqref{eqn:Gaussian-1} respectively to obtain
\begin{equation}\label{eqn:density-8c}
\int_{\mathcal{H}_{-s}}\!\!\!\!I_{k,1}(X)\mu(dX)=0,\qquad\int_{\mathcal{H}_{-s}}\!\!\!\!I_{k,2}(X)\mu(dX)=0,\ k\geq 1.
\end{equation}
Formula \eqref{eqn:gen-cut-1} now follows immediately from \eqref{eqn:gen-cut-7a}, \eqref{eqn:density-8a}, \eqref{eqn:density-8b} and \eqref{eqn:density-8c}, thus completing the proof.
\end{proof}
We now show that $\mu$ is \emph{almost invariant} under the cutoff system \eqref{eqn:infinite-mode-short-cut}. 
\begin{proposition} \label{prop:density-cut} Let $R>0$.  For every $\psi\in C^2_b(\H_{-s})$, there exists $\epsilon^R(\psi,t)>0$ such that
\begin{equation}\label{eqn:density-cut-1}
\int_{\mathcal{H}_{-s}}\!\!\!\!\PMarkov^R(t)\psi(X)\mu(dX)=\int_{\mathcal{H}_{-s}}\!\!\!\!\psi(X)\mu(dX)+\varepsilon^R(\psi,t).
\end{equation}
Furthermore, $\varepsilon^R(\psi,t)\to 0$ as $R\to\infty$.
\end{proposition}
\begin{proof} Since equation~\eqref{eqn:infinite-mode-short-cut} has a globally Lipschitz drift term, in view of Theorem 9.23 from \cite{da2014stochastic}, for every $\psi\in C^2_b(\H_{-s})$, $\PMarkov^R(t)\psi\in  C_b^2(\H_{-s})$ satisfies the Kolmogorov backward equation, namely
\begin{equation}\label{eqn:density-0}
\PMarkov^R(t)\psi(X)=\psi(X)+\int_0^t\L^R\PMarkov^R(r)\psi(X)dr.
\end{equation}
Integrating both sides on $\H_{-s}$ with respect to $\mu$ gives
\begin{equation}\label{eqn:density-1}
\int_{\mathcal{H}_{-s}}\!\!\!\!\PMarkov^R(t)\psi(X)\mu(dX)=\int_{\mathcal{H}_{-s}}\!\!\!\!\psi(X)\mu(dX)+\int_0^t\int_{\mathcal{H}_{-s}}\!\!\!\!\L^R\PMarkov^R(r)\psi(X)\mu(dX)dr.
\end{equation}
We note that Fubini's theorem was applied to switch the order of integration in the double-integral term above. Indeed, from \eqref{ineq:gen-cut-8a}-\eqref{ineq:gen-cut-8d}, we see that for all $r\in[0,t]$
\begin{equation*}
\int_{\mathcal{H}_{-s}}\!\!\!\!\left|\L^R\PMarkov^R(r)\psi(X)\right|\mu(dX)\leq c \left(\| D\PMarkov^R(r)\psi\|_\infty+\| D^2\PMarkov^R(r)\psi\|_\infty\right),
\end{equation*}
where $c>0$ is a constant independent of $R>0$. Furthermore, in view of Theorem 9.8 and Theorem 9.9 from \cite{da2014stochastic}, $\sup_{0\leq r\leq t}\|  D\PMarkov^R(r)\psi\|_\infty$ and $\sup_{0\leq r\leq t}\|  D^2\PMarkov^R(r)\psi\|_\infty$ are both finite. We thus infer that $\int_0^t\int_{\mathcal{H}_{-s}}\left|\L^R\PMarkov^R(r)\psi(X)\right|\mu(dX)dr<\infty$, which guarantees that the Fubini's Theorem is applicable. Now, since $\PMarkov(t)\psi\in C^2_b(\H_{-s})$ for all $t\geq 0$, Lemma \ref{lem:gen-cut} implies that
\begin{multline}\label{eqn:density-1a}
\int_{\mathcal{H}_{-s}}\!\!\!\!\PMarkov^R(t)\psi(X)\mu(dX)=\int_{\mathcal{H}_{-s}}\!\!\!\!\psi(X)\mu(dX)\\
+\int_0^t\int_{\mathcal{H}_{-s}}\!\!\!\!v\Phi'(x)\left(1-\theta^R(x)\right)\PMarkov^R(r)\psi(X)\mu(dX)dr.
\end{multline}
Let $\varepsilon^R(\psi,t)$ be given by
\begin{equation} \label{eqn:density-1b}
\varepsilon^R(\psi,t) := \int_0^t\int_{\mathcal{H}_{-s}}\!\!\!\!v\Phi'(x)\left(1-\theta^R(x)\right)\PMarkov^R(r)\psi(X)\mu(dX)dr.
\end{equation}
It is clear that the integrand on the above RHS is dominated by $\|\psi\|_{\infty}\left|v\Phi'(x)\right|$ and that \begin{equation*}
\|\psi\|_\infty\int_0^t\int_{\mathcal{H}_{-s}}\!\!\!\!|v\Phi'(x)|\mu(dX)dr=t\|\psi\|_\infty\int_\rbb|v|\mu_v(dv)\int_\rbb\left|\Phi'(x)\right|\mu_x(dx)<\infty,
\end{equation*}
since $\mu_v$ is Gaussian and by Assumption~\ref{cond:Phi}, $\Phi'(x)e^{-\Phi(x)}$ is integrable. We additionally note that by the construction of local solutions, $X^R(r)\to X(r)$ as $R\to\infty$ a.s. It follows that $\PMarkov^R(r)\psi(X)\to\PMarkov(r)\psi(X)$, implying $v\Phi'(x)\left(1-\theta^R(x)\right)\PMarkov^R(r)\psi(X)\to 0$, since $\theta^R(x)\to 1$. We therefore apply the Dominated Convergence Theorem to infer that $\varepsilon^R(\psi,t) \to 0$, which completes the proof.
\end{proof}

With Proposition \ref{prop:density-cut} in hand, we are ready to give the proof of Theorem~\ref{thm:density}.
\begin{proof}[Proof of Theorem \ref{thm:density}.]
By taking $R\to\infty$ on both sides of \eqref{eqn:density-cut-1}, we obtain for all $\psi\in C^2_b(\H_{-s})$
\begin{equation*}\label{eqn:density-9}
\int_{\mathcal{H}_{-s}}\!\!\!\!\PMarkov(t)\psi(X)\mu(dX)=\int_{\mathcal{H}_{-s}}\!\!\!\!\psi(X)\mu(dX).
\end{equation*}
For $\varphi\in C_b(\H_{-s})$, approximating $\varphi$ by a sequence $\left\{\psi_k\right\}\subset C^2_b(\H_{-s})$, we apply the Dominated Convergence Theorem to arrive at
\begin{equation*}\label{eqn:density-9a}
\int_{\mathcal{H}_{-s}}\!\!\!\!\PMarkov(t)\varphi(X)\mu(dX)=\int_{\mathcal{H}_{-s}}\!\!\!\!\varphi(X)\mu(dX).
\end{equation*}
The proof is complete.
\end{proof}

\section{Uniqueness of the invariant measure in the diffusive regime} \label{sec:invariant-unique} 
In order to prove uniqueness of $\mu$, we will construct an {\it asymptotic coupling} using an appropriate Girsanov shift argument, following and applying the methods and ideas developed in~\cite{weinan2001gibbsian,mattingly2002exponential,hairer2011asymptotic}.  Intuitively, this means that solutions started from different initial data have a positive probability of converging to one another as $t\rightarrow \infty$.  Theorem 1.1 of~\cite{hairer2011asymptotic} will then allow us to conclude there is only one ergodic invariant measure, thus uniqueness of $\mu$ follows by ergodic decomposition. The idea of using Girsanov shift to construct \emph{asymptotic coupling} first appeared in the work of~\cite{weinan2001gibbsian} and was later developed in~\cite{hairer2011asymptotic}. For some more recent applications of this theory to SPDEs, we refer the reader to~\cite{glatt2017unique,kulik2015generalized}. 

For the reader's convenience, we briefly explain the framework of the \emph{asymptotic coupling method} adapted to our setting, following \cite{hairer2011asymptotic,glatt2017unique}. To begin, we denote by $\Hs^\nbb$ the pathspace over $\Hs$, 
\begin{displaymath}
\Hs^\nbb=\{\Ubf:\nbb\to\Hs\}=\{\Ubf = (U_0,U_1,U_2,\dots): U_i\in\Hs\},
\end{displaymath} 
and let $\PMarkov(\Hs^\nbb\times\Hs^\nbb)$ be the set of probability measures on $\Hs^\nbb\times\Hs^\nbb$. For any two measures $M_1,\, M_2$ on $\Hs^\nbb$, we denote by $\wt{\mathcal{C}}(M_1,M_2)$ the collection of \emph{asymptotically equivalent coupling} for $M_1,\, M_2$,
\begin{equation}
\wt{\mathcal{C}}(M_1,M_2)=\{\Gamma\in\PMarkov(\Hs^\nbb\times\Hs^\nbb):\Gamma\Pi_i^{-1}<<M_i,i=1,2\},
\end{equation}
 where $\Pi_1(u,v)=u$ and $\Pi_2(u,v)=v$. For any initial condition $X_0\in\Hs$, let $\Xbf = (X_0,X(1),X(2),\dots)$ be the corresponding solution path on $\Hs^\nbb$ where $X(t)$ solves~\eqref{eqn:infinite-mode-short}. Then the law of $\Xbf$, denoted by $\delta_{X_0}\PMarkov^\nbb$, defines a probability measure on $\Hs^\nbb$. Next, we introduce the set $\mathscr{D}$ given by
 \begin{align} \label{form:coupling:D}
 \mathscr{D}=\{(\Ubf,\Vbf)\in\Hs^\nbb\times\Hs^\nbb:\lim_{n\to\infty} \|U_n-V_n\|_{-s}=0  \}.
 \end{align}
Having introduced the above, we will seek to apply the following result (c.f. Corollary~2.2 of~\cite{hairer2011asymptotic} and Corollary 2.1 \cite{glatt2017unique}).

\begin{theorem} \label{thm:unique:glatt}
If for every pair $X_0,\, \Xtil_0\in\Hs$, there exists an element $\Gamma\in\wt{C}(\delta_{X_0}\PMarkov^\nbb,\delta_{\Xtil_0}\PMarkov^\nbb)$ such that $\Gamma(\mathscr{D})>0$, then there exists at most one ergodic invariant measure for~\eqref{eqn:infinite-mode-short}.
\end{theorem}
The problem thus reduces to constructing such a coupling $\Gamma$. To this end, we introduce another process $\Xtil(t)$ on $\H_{-s}$ satisfying the following shifted version of equation~\eqref{eqn:infinite-mode-short}  
\begin{align} 
d\Xtil(t)&=L\Xtil(t)\, dt+F(\Xtil(t))\, dt+B\, dW(t)+B \, U(X(t),\Xtil(t))\textbf{1}\{t\leq\tau\}\, dt.\label{eqn:infinite-mode-coupling-b}
\end{align}
In the above, $\Xtil(0)=\Xtil_0\in \H_{-s}$, $\tau$ is a stopping time and $U(X(t),\Xtil(t))\in L^2\left(\Omega,\W\right)$ is an adapted control depending on both $\Xtil$ and the process $X$ satisfying~\eqref{eqn:infinite-mode-short} with $X(0)=X_0 \in \H_{-s}$. Here we recall that $\W$ is the auxiliary Hilbert space on which $W(t)$ evolves, \cite{da2014stochastic}. Now notice that if we set 
\begin{equation}\label{eqn:coupling-noise}
\widetilde{W}(t)=W(t)+\int_0^tU(X(r),\Xtil(r)) \textbf{1}\{r\leq\tau\}dr
\end{equation}
and the control $U$ and stopping time $\tau$ are such that, for some deterministic constant $C>0$,  
\begin{equation}\label{ineq:Girsanov-1}
\mathbb{P} \Big\{ \int_0^\infty \|U(X(t),\Xtil(t)) \textbf{1}\{t\leq\tau\}\|_\W^2 \, dt <C\Big\} =1,\end{equation}
then $W$ and $\wt{W}$ are equivalent on $C([0,\infty);\W)$. As a consequence, the processes $X$ and $\Xtil$ with $X(0)=\Xtil(0)= \Xtil_0\in \H_{-s}$ are mutually absolutely continuous on the infinitie time horizon $[0, \infty)$. As shown later in the proof of Theorem~\ref{thm:inv-measure:unique}, our \emph{coupling} $\Gamma$ is essentially the law of the pair $(X(\cdot,X_0),\Xtil(\cdot,\Xtil_0))$. However, in order for $\Gamma$ to meet the requirement $\Gamma(\mathscr{D})>0$ from Theorem~\ref{thm:unique:glatt}, by introducing the difference 
\begin{align}
\label{eqn:difference}
\Xbar(t)=X(t)-\Xtil(t)=\left(\overline{x}(t),\overline{v}(t),\overline{z_1}(t),\dots\right), 
\end{align}
we have to pick $U$ and $\tau$ such that~\eqref{ineq:Girsanov-1} is satisfied, $\P{\tau=\infty}>0$ and 
\begin{align*}
\|\Xbar(t)\|_{-s}\to 0 \, \, \text{ as } \,\, t\to\infty \,\, \text{ on the event } \,\, \{\tau=\infty\}. 
\end{align*}
Thus, reemphasizing what was discussed above, we are constructing the control $U$ and the stopping time $\tau$ such that we can drive two solutions of \eqref{eqn:infinite-mode-short} with different initial data to one another as $t\rightarrow \infty$ on a set of positive probability.  

To introduce our choice of $U$ and $\tau$, first observe that $\Xbar$ satisfies $\Xbar(0)= X_0-\Xtil_0$ and 
\begin{equation} \label{eqn:infinite-mode-coupling-c}
d\Xbar(t)=L\Xbar(t)\, dt+\big[F(X(t))-F(\Xtil(t))\big]\, dt-B \, U(X(t),\Xtil(t)) \textbf{1}\{t\leq\tau\} \, dt.
\end{equation}
Writing $X(t)= (x(t), v(t), z_1(t), \ldots)$, $\Xtil(t)=(\tilde{x}(t), \tilde{v}(t), \tilde{z}_1(t), \ldots)$ and recalling the notation~\eqref{eqn:difference}, we define for given $\lambda >0$
\begin{multline} \label{eqn:coupling}
u_0(X(t), \Xtil(t))=\frac{m}{\sqrt{2\gamma}}\Big[\left(3\lambda-\frac{\gamma}{m}\right)\overline{v}(t)+2\lambda^2\xbar(t)\\
-\frac{1}{m}\left(\Phi'(x(t))-\Phi'(\widetilde{x}(t))\right)-\frac{1}{m}\sum_{k\geq 1}\sqrt{c_k}\zbar_k(t)\Big] ,
\end{multline}
and 
\begin{align}
U(X(t), \Xtil(t))= (0, u_0(X(t), \Xtil(t)), 0, 0, \cdots). 
\end{align}
Note that $B U$ only possibly enacts control over the velocity difference $\bar{v}(t) = v(t) - \tilde{v}(t)$, and this is essentially done to gain control over nonlinear difference $\Phi'(x)- \Phi'(\tilde{x})$.  

For a given $\kappa >0$, we define the stopping time $\tau=\tau(\kappa)$ by   
\begin{equation}\label{eqn:stop-time}
\tau=\inf_{t\geq 0}\left\{\int_0^t |u_0(X(s), \Xtil(s))|^2 \, ds\geq \kappa\right\}.
\end{equation}

With these choices, note that on the event $\{ t< \tau\}$, $\Xbar(t)=(\bar{x}(t), \bar{v}(t), \bar{z}_1(t), \ldots)$ satisfies the following system of equations
\begin{align}\label{eqn:infinite-coupling}
\frac{d \xbar(t)}{dt} &=\ \vbar(t), & \xbar(0)&=\xbar_0,\\
\nonumber \frac{d \vbar(t)}{dt}&=-3\lambda\vbar(t)-2\lambda^2\xbar(t), &\vbar(0)&=\vbar_0,\\
\nonumber \frac{d \zbar_k(t)}{dt}&=-\lambda_k \zbar_k(t)+\sqrt{c_k}\vbar(t),&\zbar_k(0)&=(\zbar_k)_0 .\end{align}
Intuitively, the coefficient $\lambda>0$ will be picked so that that $\|\Xbar(t)\|_{\H_{-s} }\rightarrow 0$ as $t\rightarrow \infty$ on the event $\{ \tau = \infty\}$.   Hence the control induces the requisite dissipation, but we still need to see that we can pick $\kappa >0$ so that ~\eqref{ineq:Girsanov-1} is satisfied and $\mathbb{P} \{ \tau=\infty\} >0$.  Before turning to this issue, we make the following remark.

\begin{remark} (a) There is a significantly flexibility in the choice of $u_0$ in \eqref{eqn:coupling}. One can of course choose other formulas for the coefficients of $\xbar(t)$ and $\vbar(t)$ in \eqref{eqn:coupling} as long as $\| \Xbar(t)\|_{-s}\rightarrow 0$ as $t\rightarrow \infty$ on $ \{ \tau= \infty\}$. 

(b). The appearance of $u_0$ requires the drag constant $\gamma$ be strictly positive. We note that for well-posedness (cf. Proposition~\ref{prop:sol'n}) and the existence of invariant measures (cf. Theorem~\ref{thm:density}), $\gamma$ can be zero.

\end{remark}

With these observations, we state the following proposition which outlines the needed details to deduce unique ergodicity.  
\begin{proposition}\label{prop:asymp-coupling} Under the Hypothesis of Theorem~\ref{thm:inv-measure:unique} and recalling $m,\gamma>0$ from~\eqref{eqn:GLE-Markov} and $\alpha,\,\beta$ as in~\eqref{c-k}, let $\lambda>0$ be as in~\eqref{eqn:coupling} and $\kappa>0$ as in~\eqref{eqn:stop-time}. Then there exist $\lambda=\lambda(\alpha,\beta)>0$, $\kappa=\kappa(X_0,\Xtil_0,\gamma,m,\alpha,\beta)>0$ such that $\tau=\tau(\kappa)$ and $U$ are such that 
\begin{itemize}
\item[(a)]  Condition~\eqref{ineq:Girsanov-1} is satisfied. 
\item[(b)] $\|\Xbar(t)\|_{-s}\rightarrow 0$ as $t\rightarrow\infty$ on $\{\tau=\infty\}$. 
\item[(c)] $\P{\tau=\infty}>0$.
\end{itemize} 
\end{proposition}

Before presenting the proof of Proposition~\ref{prop:asymp-coupling}, we now show how one can deduce unique ergodicity of~\eqref{eqn:infinite-mode-short} by combining Proposition~\ref{prop:asymp-coupling} and Theorem~\ref{thm:unique:glatt}, (see \cite{hairer2011asymptotic,  glatt2017unique} for further details).
\begin{proof}[Proof of Theorem~\ref{thm:inv-measure:unique}] In view of Proposition~\ref{prop:asymp-coupling} (a), the Novikov's condition is verified
\begin{align*}
\Enone\Big[\exp{\int_0^\infty \|U(X(t),\Xtil(t)) \textbf{1}\{t\leq\tau\}\|_\W^2 \, dt}\Big]<e^C.
\end{align*}
The process $\wt{W}(t)$ defined in~\eqref{eqn:coupling-noise} is thus equivalent to the Wiener process $W(t)$ on $C([0,\infty);\W)$ by Girsanov's theorem, (Theorem 10.4 from~\cite{da2014stochastic}). It follows that the process $\wt{X}(\cdot,\wt{X}_0)$ solving~\eqref{eqn:infinite-mode-coupling-b} is absolutely continuous to the process $X(\cdot,\Xtil_0)$ on $C([0,\infty);\Hs)$, (see e.g. \cite{revuz2013continuous}). It follows that the law $\Gamma$ induced by 
\begin{displaymath}
\big\{\big( X(nt,X_0),\Xtil(nt,\Xtil_0) \big):n=0,1,2,\dots\big\}
\end{displaymath} belongs to $\wt{\mathcal{C}}(\delta_{X_0}\PMarkov^\nbb,\delta_{\Xtil_0}\PMarkov^\nbb)$. In addition, Proposition~\ref{prop:asymp-coupling} (b) and (c) imply that $\Gamma(\mathscr{D})>0$ where $D$ is given by~\eqref{form:coupling:D}. We therefore conclude unique ergodicity by virtue of Theorem~\ref{thm:unique:glatt}, thus completing the proof.
\end{proof}
We now turn to the proof of Proposition \ref{prop:asymp-coupling}. Parts (a) and (b) essentially follow by construction.  Establishing part (c), however, requires a bit more work.  To complete the proof, we need a crucial estimate on the potential $\Phi$ which relies on Lyapunov methods.  To this end, for $N \in \nbb$ and $s \in \rbb$, we introduce $\Theta:\H_{-s}\to\rbb$ defined by 
\begin{equation} \label{eqn:Lyapunov:diffusive}
\Theta(X ; s, N)=\frac{1}{m}\Phi(x)+\frac{1}{2} v^2+\frac{1}{2m}\sum_{k=1}^N z_k^2+\frac{1}{2}\sum_{k > N}k^{-2s}z_k^2.
\end{equation}
In the diffusive regime, it turns out that $\Theta(X;s,N)$ can be chosen such that it satisfies a Lyapunov bound that is stronger than the bound on $\Psi$ from Proposition~\ref{prop:Ito-Lyapunov}. That is:
\begin{proposition}\label{prop:Ito-Lyapunov:diffusive}
Let $\Theta(X;s,N)$ be defined as in \eqref{eqn:Lyapunov:diffusive}.  Then, under Assumptions \ref{cond:Phi} and Condition \emph{(D)} of Assumption \ref{cond:sol'n}, there exists $N=N(m,\gamma,\alpha,\beta,s) \in \nbb$ sufficiently large such that, for some $ a > 0$, $\Theta(X):=\Theta(X;s,N)$ satisfies
\begin{displaymath}
\sup_{X\in\H_{-s}}\L \Theta(X ) \leq a.
\end{displaymath}
\end{proposition}
\noindent Proposition~\ref{prop:Ito-Lyapunov:diffusive} will be established at the end of the section, but the proof follows a similar line of reasoning to that employed in the proof of Proposition~\ref{prop:Ito-Lyapunov}.  

\begin{remark} (a) In Assumption~\ref{cond:sol'n}, the diffusive regime (D) requires $\alpha >1$ as opposed to $\alpha \in (0,1)$ and $\alpha =1$ in, respectively, the subdiffusive (SD) and critical (C) regimes.  Recalling $c_k,\,\lambda_k$ from~\eqref{c-k}, the condition $\alpha>1$ is needed so that the infinite sum 
\begin{displaymath}
\sum_{k\geq 1}\frac{c_k}{2\lambda_k k^{-2s}}=\sum_{k\geq 1}k^{-1-(\alpha-1)\beta+s}
\end{displaymath} 
converges, as shown later in \eqref{ineq:coupling-18} and \eqref{eqn:Ito-Lyapunov:diffusive:5}.  This convergence is critically employed in the proofs of Proposition~\ref{prop:asymp-coupling} (c) and Propositon~\ref{prop:Ito-Lyapunov:diffusive}. 

(b) The asymptotic behavior of $\lambda_k$ as $k\to\infty$ presents a barrier to obtaining a stronger Lyapunov bound of the form
\begin{displaymath}
\label{eqn:geoergbound}
\L \Theta(X ) \leq -c\Theta(X) + a
\end{displaymath}
in the proof of Proposition~\ref{prop:Ito-Lyapunov:diffusive}, where $c>0$ is a constant.  The above inequality, however, can be readily achieved in the finite-dimensional system~\eqref{eqn:finite-mode}, see \cite{ottobre2011asymptotic}.  With the appropriate support properties of the diffusion, such a bound implies geometric ergodicity.  However, because we cannot see immediately why~\eqref{eqn:geoergbound} holds in our infinite-dimensional system, suggests that perhaps the system relaxes to equilibrium slower than an exponential rate.          
\end{remark}

By combining the previous Proposition with the exponential martingale inequality, we obtain the following corollary.    
\begin{corollary} \label{cor:exp-g} Under Assumptions \ref{cond:Phi} and Condition \emph{(D)} of Assumption~\ref{cond:sol'n}, let $X(t)=(x(t), v(t),  \ldots)$ be the solution of \eqref{eqn:GLE-Markov} with initial condition $X_0\in \H_{-s}$. Let $\Theta$ be the Lyapunov function defined in \eqref{eqn:Lyapunov:diffusive}. Then there exists $\varepsilon=\varepsilon(m,\gamma,\alpha,\beta)>0$ such that for every $\eta,r>0$, 
\begin{equation} \label{ineq:exp-g-1}
\Pnone \Big\{\sup_{t\geq 0}\frac{e^{-\eta t}\Phi(x(t))}{m}-\Theta(X_0)-\frac{a}{\eta}\geq r \Big\} \leq e^{-\varepsilon r},
\end{equation}
where $a$ is as in the statement of Proposition \ref{prop:Ito-Lyapunov:diffusive}.
\end{corollary}
\noindent The proof of Corollary~\ref{cor:exp-g} will also be given at the end of this section.  

We now conclude Proposition~\ref{prop:asymp-coupling} assuming the previous two results.  
\begin{proof}[Proof of Proposition \ref{prop:asymp-coupling}]  We begin by showing part (a) of the result.  In view of formulas \eqref{eqn:Wiener} and \eqref{eqn:B}, the norm of the control $U(t)$ in $\W$ satisfies
\begin{equation*}
\|U(t)\|_\W^2=u_0(t)^2 .  
\end{equation*}
It thus follows by definition of $\tau$ that 
\begin{equation*}
\int_0^\infty \|U(X(t),\Xtil(t))\textbf{1}\{t\leq\tau\}\|_\W^2 \, dt = \int_0^\infty |u_0(t)|^2 \textbf{1}\{t\leq\tau\}\, dt=\kappa\qquad \mathbb{P}-\text{almost surely}.
\end{equation*}
Applying Theorem 10.4 from \cite{da2014stochastic} finishes the proof of part~(a).  

To conclude part~(b), for $t\leq \tau $ one can readily verify that the exact solution of \eqref{eqn:infinite-coupling} is given by 
\begin{equation*}
\begin{aligned}
\xbar(t)&=\Big(2\xbar_0+\frac{\vbar_0}{\lambda}\Big)e^{-\lambda t}-\Big(\xbar_0+\frac{\vbar_0}{\lambda}\Big)e^{-2\lambda t}\\
\vbar(t)&=-\left(2\lambda\xbar_0+\vbar_0\right)e^{-\lambda t}+2\left(\lambda\xbar_0+\vbar_0\right)e^{-2\lambda t}\\
\zbar_k(t)&=e^{-\lambda_k t}\Big[\left(\zbar_k\right)_0 +\sqrt{c_k}\int_0^t e^{\lambda_k r} \bar{v}(r) dr\Big].
\end{aligned}
\end{equation*}
From this, it follows that
\begin{equation}
\label{ineq:coupling-4}
\left|\xbar(t)\right|\leq C_1 e^{-\lambda t},\qquad \left|\vbar(t)\right|\leq C_2 e^{-\lambda t},
\end{equation}
where 
\begin{equation*} 
C_1:=3|\xbar_0|+\frac{3|\vbar_0|}{\lambda},\qquad C_2:=4\lambda|\xbar_0|+4|\vbar_0|.
\end{equation*}
Combining these bounds, we thus obtain the following bound on $\zbar_k(t)$
\begin{equation*}
\left|\zbar_k(t)\right|\leq e^{-\lambda_k t}\Big[\left|\left(\zbar_k\right)_0\right| +C_2\sqrt{c_k}\int_0^t e^{(\lambda_k-\lambda) r} dr\Big].
\end{equation*}
Choosing $\lambda=\lambda_1+1$ note that $\lambda-\lambda_k\geq 1$ for all $k\geq 1$ since $\lambda_k\downarrow 0$. With this choice of $\lambda$, it follows from the inequality above that for all $k\geq 1$,
\begin{equation}
\label{ineq:coupling-7}
\left|\zbar_k(t)\right|\leq e^{-\lambda_k t}\left(\left|\left(\zbar_k\right)_0\right| +C_2\sqrt{c_k}\right),
\end{equation}
and hence by Young's inequality,
\begin{equation*}
\zbar_k(t)^2\leq  2e^{-2\lambda_k t}\big[\left(\zbar_k\right)_0^2 +C_2^2c_k\big] .
\end{equation*}
Thus putting it all together we find that
\begin{align*}
\|\Xbar(t)\|_{-s}^2&=\xbar(t)^2+\vbar(t)^2+\sum_{k\geq 1}k^{-2s}\zbar_k(t)^2\\
&\leq \left(C_1^2+C_2^2\right)e^{-\lambda t}+2\sum_{k\geq 1}k^{-2s}\left(\zbar_k\right)_0^2e^{-2\lambda_k t}+2C_2^2\sum_{k\geq 1}k^{-2s}e^{-2\lambda_k t}c_k.
\end{align*}
Thus on the event $\{\tau=\infty\}$, it is now evident that $\|\Xbar\|_{-s}^2\to 0$ as $t\to\infty$ by applying the Monotone Convergence Theorem.

Turning to part (c) of the result, for any $R>0$ consider the event $E_R$ given by 
\begin{equation} \label{eqn:coupling-9}
E_{R}=\Big\{\sup_{t\geq 0}\frac{e^{-\lambda t/2q}\Phi(x(t))}{m}-\Theta(X_0)-\frac{2qa}{\lambda}< R\Big\},
\end{equation}
where $q$ is the constant from Asumption \ref{cond:Phi-1}. In view of Corollary~\ref{cor:exp-g} with $\eta=\lambda/2q$, $E_{R}$ has positive probability provided $R=R(\gamma,m,\alpha,\beta)>0$ is  sufficiently large. We first claim that on $E_R$, 
\begin{align*}
\int_0^\infty \|U(t)\|_{\W}^2 \, dt \,\,  \text{ is bounded almost surely}.
\end{align*}
To see this, recall by definition of the control $U$ that 
\begin{align*}
\int_0^\infty \|U(t)\|_{\W}^2dt=\int_0^\infty |u_0(t)|^2dt. 
\end{align*}
Thus estimating $u_0(t)^2$, from~\eqref{eqn:coupling} we have 
\begin{equation*}
\begin{aligned}
u_0(t)^2&\leq \frac{2 m^2}{\gamma} \Big[\left(3\lambda-\frac{\gamma}{m}\right)^2\vbar(t)^2+4\lambda^4\xbar(t)^2\\
&\qquad\qquad\qquad\qquad\,\,+\frac{1}{m^2}\Big(\sum_{k\geq 1}\sqrt{c_k}|\zbar_k(t)|\Big)^2+\frac{1}{m^2}\left|\Phi'(x(t))-\Phi'(\widetilde{x}(t))\right|^2\Big]\\
&= I_1(t)+I_2(t)+I_3(t)+I_4(t).
\end{aligned}
\end{equation*}
For $I_1(t)+I_2(t)$, apply~\eqref{ineq:coupling-4} to find
\begin{align}\label{ineq:coupling-11}
I_1(t)+I_2(t) &\leq \frac{2m^2}{\gamma}  \Big[\left(3\lambda-\frac{\gamma}{m}\right)^2C_2^2+4\lambda^4C_1^2\Big]e^{-2\lambda t}=C_3 e^{-2\lambda t},
\end{align}
where $C_3:= \frac{2m^2}{\gamma}  \left[\left(3\lambda-\frac{\gamma}{m}\right)^2C_2^2+4\lambda^4C_1^2\right]$. For $I_3(t)$, employ~\eqref{ineq:coupling-7} to see that  
\begin{align*}
I_3(t)&\leq \frac{2}{\gamma} \Big(\sum_{k\geq 1}\sqrt{c_k}e^{-\lambda_k t}\big(|(\zbar_k)_0| +C_2\sqrt{c_k}\big)\Big)^2\\&\leq \frac{4}{\gamma}\Big(\sum_{k\geq 1}\sqrt{c_k}e^{-\lambda_k t}|(\zbar_k)_0|\Big)^2+\frac{4C_2^2}{\gamma}\Big( \sum_{k\geq 1}c_ke^{-\lambda_k t}\Big)^2\\
&\leq \frac{4\Xbar_0^2}{\gamma}\sum_{k\geq 1}\frac{c_k e^{-2\lambda_k t}}{k^{-2s}}+\frac{4C_2^2}{\gamma}K(t)^2,
\end{align*}
where the last inequality follows by Cauchy-Schwarz inequality since $K(t)=\sum_{k\geq 1}c_ke^{-\lambda_k t}$ by definition. Lastly, to estimate $I_4(t)$, Assumption \ref{cond:Phi-1} and~\eqref{ineq:coupling-4} together imply that
\begin{align*}
I_4(t) &\leq \frac{4}{\gamma}\xbar(t)^2\big(f\left(\xbar(t)\right)^2+\Phi(x(t))^{2q}\big)\leq \frac{4C_1^2e^{-2\lambda t}}{\gamma}\Big(\sup_{|y|\leq C_1}f\left(y\right)^2+\Phi(x(t))^{2q}\Big).
\end{align*}
Now on the event $E_{R}$, we note that
\begin{equation*}
\sup_{t\geq 0}e^{-\lambda t}\Phi(x(t))^{2q} <\Big(m \Theta(X_0)+\frac{2qma}{\lambda}+m R\Big)^{2q}=:C_4.
\end{equation*}
Hence
\begin{equation*}
\begin{aligned}
I_4(t) &\leq \frac{4C_1^2e^{-\lambda t}}{\gamma}\Big(e^{-\lambda t}\sup_{|y|\leq C_1}f\left(y\right)^2+\sup_{t\geq 0}e^{-\lambda t}\Phi(x(t))^{2q}\Big)\\
&\leq\frac{4C_1^2}{\gamma}\Big(\sup_{|y|\leq C_1}f\left(y\right)^2+C_4\Big)e^{-\lambda t}\\
&=C_5 e^{-\lambda t}.
\end{aligned}
\end{equation*}
Combining these bounds for $I_1, I_2, I_3, I_4$ shows that on $E_R$,
\begin{equation} \label{ineq:coupling-18}
\begin{aligned}
\int_0^\infty u_0(t)^2 dt &\leq \int_0^\infty I_1(t)+I_2(t) +I_3(t) +I_4(t) \, dt \\
&\leq \int_0^\infty C_3 e^{-2\lambda t}+\frac{4\Xbar_0^2}{\gamma}\sum_{k\geq 1}\frac{c_k e^{-2\lambda_k t}}{k^{-2s}}+\frac{4C_2^2}{\gamma}K(t)^2+C_5e^{-\lambda t} \, dt\\
&= \frac{4\Xbar_0^2}{\gamma}\sum_{k\geq 1}\frac{c_k}{2\lambda_k k^{-2s}}+\int_0^\infty C_3 e^{-2\lambda t}+\frac{4C_2^2}{\gamma}K(t)^2+C_5e^{-\lambda t} dt.
\end{aligned}
\end{equation}
We invoke Assumption \ref{cond:sol'n} again to see that $\sum_{k\geq 1}\frac{c_k}{2\lambda_k k^{-2s}}<\infty$. Furthermore, in view of~\eqref{lim:K}, $K(t)^2\sim t^{-2\alpha}$ as $t\to\infty$, implying $K(t)^2$ is integrable. We thus infer from \eqref{ineq:coupling-18} a constant $C_6=C_6(X_0,\Xtil_0,\gamma,m,\alpha,\beta)>0$ such that on $E_R$, 
\begin{equation} \label{ineq:coupling-19}
\int_0^\infty u_0(t)^2 dt \leq C_6.
\end{equation}
Finally, we choose $\kappa>C_6$ in the definition of $\tau=\tau(\kappa)$ forcing $E_R\subset\{\tau=\infty\} $. We therefore, conclude that $\P{\tau_\kappa=\infty}>0$. The proof is thus complete. 
\end{proof}

We now finish this section by giving the proofs of Proposition~\ref{prop:Ito-Lyapunov:diffusive} and Corollary~\ref{cor:exp-g}.  
\begin{proof}[Proof of Proposition~\ref{prop:Ito-Lyapunov:diffusive}]
We have 
\begin{align*}
\L \Theta &= -\frac{\gamma}{m}v^2-\frac{1}{m}\sum_{k=1}^N\lambda_k z_k^2-\sum_{k>N}\lambda_k k^{-2s}z_k^2-\frac{1}{m}\sum_{k>N}\sqrt{c_k}z_k v\\
&\qquad+\sum_{k>N}\sqrt{c_k}k^{-2s}z_k v+ \frac{\gamma}{m^2}+\frac{1}{m}\sum_{k=1}^N\lambda_k+\sum_{k>N}\lambda_k k^{-2s} .  
\end{align*}
Young's inequality combined with Cauchy-Schwarz inequality then gives 
\begin{align*}
\frac{1}{m}\sum_{k>N}\sqrt{c_k}z_k v\leq \frac{\gamma}{4m}v^2+\frac{1}{\gamma m}\sum_{k>N}\frac{c_k}{k^{-2s}\lambda_k} \sum_{k>N}k^{-2s}\lambda_kz_k^2
\end{align*} 
and
\begin{align*}
\sum_{k>N}k^{-2s}\sqrt{c_k}z_k v\leq \frac{\gamma}{4m}v^2+\frac{m}{\gamma }\sum_{k>N}\frac{k^{-2s}c_k}{\lambda_k}\sum_{k>N}k^{-2s}\lambda_k z_k^2.
\end{align*} 
Combining the previous two inequalities with the first we obtain
\begin{equation}\label{ineq:Ito-Lyapunov:diffusive:4}
\begin{aligned}
\L \Theta &\leq -\frac{\gamma}{2m}v^2-\frac{1}{m}\sum_{k=1}^N\lambda_k z_k^2-a_1\sum_{k>N}k^{-2s}\lambda_kz_k^2 +a,
\end{aligned}
\end{equation}
where 
\begin{equation}\label{eqn:Ito-Lyapunov:diffusive:5}
a:=\frac{\gamma}{m^2}+\frac{1}{m}\sum_{k=1}^N\lambda_k+\sum_{k>N}\lambda_k k^{-2s} ,\quad a_1:=1-\frac{1}{\gamma m} \sum_{k>N}\frac{c_k}{k^{-2s}\lambda_k}-\frac{m}{\gamma }\sum_{k>N}\frac{k^{-2s}c_k}{\lambda_k}.
\end{equation}
We invoke Condition (D) of Assumption \ref{cond:sol'n} again to see that
\begin{equation*}
\begin{aligned}[l]
\sum_{k\geq 1}\lambda_k k^{-2s}=\sum_{k\geq 1}\frac{1}{k^{\beta+2s}}<\infty,&\qquad
\sum_{k\geq 1}\frac{c_k}{k^{-2s}\lambda_k}=\sum_{k\geq 1}\frac{1}{k^{1+(\alpha-1)\beta-2s}}<\infty,\\
\text{and}\qquad\sum_{k\geq 1}\frac{k^{-2s}c_k}{\lambda_k}&=\sum_{k\geq 1}\frac{1}{k^{1+(\alpha-1)\beta+2s}}<\infty,
\end{aligned}
\end{equation*}
which implies that $a<\infty$ and that $N$ can be chosen large enough such that $0<a_1<\infty$. We therefore conclude $\L \Theta \leq a$, which is the desired inequality.
\end{proof}

\begin{proof}[Proof of Corollary~\ref{cor:exp-g}]. Fix $\eta>0$ and apply Ito's Formula to $e^{-\eta t}\Theta(X(t))$ to find
\begin{equation}\label{eqn:exp-g-2}
d (e^{-\eta t}\Theta(X(t)))=-\eta e^{-\eta t}\Theta(X(t)) \, dt + e^{-\eta t}\mathcal{L}\Theta(X(t))\, dt + d M_\eta (t)
\end{equation}
where the martingale $M_\eta$ satisfies 
\begin{align*} \label{eqn:exp-g-3}
d M_\eta(t)& =e^{-\eta t}\frac{\sqrt{2\gamma}}{m}v(t)\, dW_0(t)+\frac{e^{-\eta t}}{m}\sum_{k=1}^N \sqrt{2\lambda_k}z_k(t)\, dW_k(t)\\
&\qquad +e^{-\eta t}\sum_{k>N}\sqrt{2\lambda_k}k^{-2s}z_k(t)\, dW_k(t).
\end{align*}
Note also that the quadratic variation process $\la M_\eta \ra$ has 
\begin{equation}\label{eqn:exp-g-4}
d \la M_\eta\ra (t) =\frac{2\gamma e^{-2\eta t} }{m^2}v(t)^2 \, dt +\frac{2 e^{-2\eta t} }{m^2}\sum_{k=1}^N \lambda_k z_k(t)^2 \, dt +2 e^{-2\eta t} \sum_{k>N}\lambda_k k^{-4s}z_k(t)^2\, dt.
\end{equation}
We recall from \eqref{ineq:Ito-Lyapunov:diffusive:4} that 
\begin{equation} \label{ineq:exp-g-5}
\begin{aligned}
\L \Theta(X(t))&\leq -\frac{\gamma}{2m}v(t)^2-\frac{1}{m}\sum_{k=1}^N\lambda_k z_k(t)^2-a_1 \sum_{k>N}k^{-2s}\lambda_kz_k(t)^2+a,
\end{aligned}
\end{equation}
where $a,\,a_1$ are defined in \eqref{eqn:Ito-Lyapunov:diffusive:5}. Combining \eqref{eqn:exp-g-2}, \eqref{eqn:exp-g-4} and \eqref{ineq:exp-g-5}, for every $\varepsilon>0$ we obtain the estimate
\begin{equation*}\label{ineq:exp-g-7}
\begin{aligned}
d\left(e^{-\eta t}\Theta(X(t))\right)  &\leq  ae^{-\eta t}dt + d M_\eta (t)-\frac{\varepsilon}{2}d \la M_\eta\ra(t)\\ 
&\qquad -e^{-2\eta t}\Big[\frac{\gamma}{2m}v(t)^2+\frac{1}{m}\sum_{k=1}^N\lambda_k z_k(t)^2+a_1\sum_{k>N}k^{-2s}\lambda_kz_k(t)^2\\
&\qquad-\frac{\varepsilon}{2}\Big(\frac{\gamma}{m^2}v(t)^2+\frac{1}{m^2}\sum_{k=1}^N \lambda_k z_k(t)^2 +\sum_{k>N}\lambda_k k^{-4s}z_k(t)^2 \Big)\Big]dt.
\end{aligned}
\end{equation*}
By choosing $\varepsilon=\varepsilon(\gamma,m,\alpha,\beta)>0$ smaller if necessary, the bracket term on the above RHS is nonpositive. Hence
\begin{equation*}\label{ineq:exp-g-8}
d(e^{-\eta t}\Theta(X(t)))\leq  ae^{-\eta t}dt + d M_\eta (t)-\frac{\varepsilon}{2}d\la M_\eta\ra(t).
\end{equation*}
Integrating with respect to $t$ we find
\begin{align*}\label{ineq:exp-g-9}
e^{-\eta t}\Theta(X(t))-\Theta(X_0) &\leq  \int_0^\infty ae^{-\eta r}dr + M_\eta (t)-\frac{\varepsilon}{2}\la M_\eta\ra(t)=\frac{a}{\eta}+M_\eta (t)-\frac{\varepsilon}{2}\la M_\eta\ra(t).
\end{align*}
Since $\Theta(X(t))\geq \Phi(X(t))/m$ by the definition of $\Theta(X)$, we infer that
\begin{equation*}\label{ineq:exp-g-10}
\frac{e^{-\eta t}\Phi(x(t))}{m}-\Theta(X_0)-\frac{a}{\eta}\leq M_\eta (t)-\frac{\varepsilon}{2}\la M_\eta\ra(t).
\end{equation*}
Invoking the exponential martingale inequality we obtain 
\begin{align*}\label{ineq:exp-g-11}
\Pnone \Big\{\sup_{t\geq 0}\frac{e^{-\eta t}\Phi(x(t))}{m}-\Theta(X_0)-\frac{a}{\eta}\geq r\Big\}&\leq\Pnone \Big\{ \sup_{t\geq 0}\left[M_\eta (t)-\frac{\varepsilon}{2}\la M_\eta\ra(t)\right]\geq r\Big\}\\
&\leq e^{-\varepsilon r},
\end{align*}
thus completing the proof.  
\end{proof}

\section{discussion} \label{sec:discussion} We have rigorously studied the GLE in a potential well $\Phi$ with a power-law decay memory $K(t)$, i.e. $K(t)\sim t^{-\alpha}$, $\alpha>0$ as $t\to\infty$. Using a Mori-Zwanzig approach, when the memory $K$ can be written as an infinite sum of exponentials, we represent the non-linear GLE as an infinite-dimensional Markovian system. With nominal conditions on the potential $\Phi$, we show that for every $\alpha>0$, this Markovian system is well-posed in suitable spaces and admits an invariant measure. Moreover, using an asymptotic coupling argument, the system is shown to be uniquely ergodic when $\alpha>1$. The problem of unique ergodicity remains open when $\alpha\in(0,1]$. 

A related research topic that is of direct interest is to establish the convergence rate to stationarity. Due to the memory's power-law decay, one might conjecture that this system does not approach the invariant measure with an exponential rate (commonly called geometric ergodicity). That is, it is conceivable that there is a unique invariant measure, but the approach is algebraic instead. Current methods for proving an algebraic rate of convergence to stationarity rely on finding a type of Lyapunov function that is currently unknown for this system.

Lastly, we would like to touch on the term ``ergodicity breaking,'' which has appeared in the physics literature in connection with models of anomalous subdiffusion \cite{he2008random,lubelski2008nonergodicity}. In particular, there are claims that solutions to the generalized Langevin equation in a quadratic potential  can break ergodicity in the sense that a popular expression for a particle time-average does not match a stationary population's ensemble average \cite{jeon2010fractional, jeon2012inequivalence}. It is important to point out though that the time average used in these papers, sometimes called the pathwise mean-squared displacement,
\begin{equation} \label{eqn:msd}
\frac{1}{T-\Delta}\int_0^{T-\Delta} (x(s+ \Delta)- x(s))^2 \, ds
\end{equation}
is a ``sliding window'' average of squared-displacements and not equivalent to the time-average formula \eqref{eqn:ergodic-theorem} used in the mathematical literature on ergodic theory. Notably, the authors of references \cite{jeon2010fractional} and \cite{jeon2012inequivalence} are able to show that the inequivalence of \eqref{eqn:msd} with ensemble averages occurs even for the Ornstein-Uhlenbeck SDE (Brownian motion in a quadratic potential), which is well-known to be a (geometrically) ergodic process in the sense of Birkoff. Therefore, the results and conjectures discussed under the heading of ``ergodicity breaking'' in the physics literature do not necessarily align with the results and conjectures we present here. Having said that, the MSD formula \eqref{eqn:msd} is an essential tool in the particle tracking literature and we believe it is an interesting and unsolved question as to why this time-average fails to match ensemble averages as one might expect for an ergodic process.

\section*{Acknowledgement}The authors would like to thank Jonathan Mattingly and Pete Kramer for fruitful discussions on the topic of this paper.  All authors are grateful for support from the mathematics departments at Tulane University and Iowa State University as well as for support through grants NSF DMS-1644290 (SAM and HDN), NSF DMS-1612898 and DMS-1855504 (DPH), and NSF DMS-1816551, DMS-1522616, and Simons Foundation grant 515990 (NEGH).



\bibliographystyle{plain}
\bibliography{nonlinear-gle}

\end{document}